\documentclass[reqno]{amsart}

\usepackage{a4wide}
\usepackage{paralist}
 \usepackage{epsfig}
 \usepackage{epstopdf}
\usepackage{graphics,graphicx}
\usepackage[english]{babel}
\usepackage{amsfonts,amsthm,amssymb,amsmath,mathrsfs}
 \usepackage[colorlinks=true]{hyperref}
\hypersetup{urlcolor=blue, citecolor=red}

\newtheorem{thm}{Theorem}[section]
\newtheorem{cor}[thm]{Corollary}
\newtheorem{lem}[thm]{Lemma}
\newtheorem{prop}[thm]{Proposition}
\theoremstyle{definition}

\theoremstyle{remark}
\newtheorem{rem}[thm]{Remark}
\numberwithin{equation}{section}
\newcommand{\Real}{\mathbb{R}}

\title[Two simple criterion to obtain controllability and stabilization]{Two simple criterion to obtain exact controllability and stabilization of a linear family of dispersive PDE's on a Periodic Domain}

	\author{Francisco J. Vielma Leal}\address{IMECC-UNICAMP, Rua Sérgio Buarque de Holanda, 651
	13083-859, Campinas, SP, Brasil}
\email{vielma@ime.unicamp.br;\;fvielmaleal7@gmail.com}

\author{Ademir Pastor }\address{IMECC-UNICAMP, Rua Sérgio Buarque de Holanda, 651, 13083-859, Campinas, SP, Brasil}
\email{apastor@ime.unicamp.br}


\begin{document}
%%-----------------------------
%%      the top matter
%%-----------------------------

\begin{abstract}
 In this work, we use the classical moment method to find a practical and simple criterion to determine  if a family of linearized Dispersive equations on a periodic domain is exactly controllable and exponentially stabilizable with any given decay rate in $H_{p}^{s}(\mathbb{T})$ with $s\in \mathbb{R}.$ We apply these results to prove that the linearized Smith equation, the linearized dispersion-generalized Benjamin-Ono equation,  the linearized fourth-order Schr\"odinger equation, and  the Higher-order Schr\"odinger equations are exactly controllable and exponentially stabilizable with any given decay rate in $H_{p}^{s}(\mathbb{T})$ with $s\in \mathbb{R}.$
\end{abstract}

\maketitle
%%-----------------------------
%%      your text
%%-----------------------------
\section{Introduction}

In this work, we consider a family of linear one-dimensional dispersive equations on the periodic domain $\mathbb{T}:=\mathbb{R}/(2\pi \mathbb{Z})$, and investigate its control properties from the point of view of distributed control.   Specifically, we consider the family of  equations
\begin{equation}\label{BO}
    \partial_{t}u- \partial_{x}\mathcal{A}u=f(x,t), \;\;\;\;x\in \mathbb{T},\;\;t\in \mathbb{R},
\end{equation}
where $u=u(x,t)$ denotes a real or complex-valued function of two real variables $x$ and $t,$ the forcing term $f=f(x,t)$ is added to the equation as a control input supported in a given open set $\omega\subset \mathbb{T}$, and $\mathcal{A}$ denotes a linear Fourier multiplier operator.  We assume that the multiplier $\mathcal{A}$ is of order $r-1,$ for some $r\in  \mathbb{R},$ with $r \geq 1$, that is, the symbol $a:\mathbb{Z} \rightarrow \mathbb{R}$ is given by
\begin{equation}\label{dispeop}
    \widehat {\mathcal{A}u}(k):=a(k)\widehat{u}, \;\;k\in \mathbb{Z},
\end{equation}
where $\widehat{u}$ stands for the Fourier transform of $u$ (see \eqref{foutrans}), and
\begin{equation}\label{limite}
    |a(k)|\leq C |k|^{r-1}, \qquad|k|\geq k_0,
\end{equation}
for some $k_0\geq0$ and some positive constant $C$.

Equation \eqref{BO} encompass a wide class of linear dispersive equations. For instance, the well-known linearized Korteweg-de Vries equation ($\mathcal{A}=-\partial_{x}^2$), the Schr\"odinger equation ($\mathcal{A}=i\partial_{x}$), the Benjamin-Ono equation ($\mathcal{A}=\mathcal{H}\partial_{x}$, where $\mathcal{H}$ stands for the Hilbert transform), and the Benjamin equation ($\mathcal{A}=\partial_{x}^2+\alpha\mathcal{H}\partial_{x}$, where $\alpha$ is a positive constant).
In the literature there is a wide range of references studing
controllability and stabilization properties  of linear and nonlinear dispersive equations. Specifically, for the Korteweg-de Vries equation  (KdV) equation, the results regarding  controllability and stabilization can be found in \cite{14,10, Zhang 1, Russell and Zhang, Rosier 1, Coron Crepau, Menzala Vasconcellos Zuazua, Rosier and Zhang 2}. For the
Schr\"odinger equation we refer the reader to \cite{Laurent Camille,  Laurent, Dehman Gerard Lebeau, Rosier Lionel Zhang, Rosier Lionel Zhang 2}. Also, the  study on the controllability and stabilization for the Benjamin and Benjamin-Ono (BO) equations  have received attention in the last decade, see \cite{Manhendra and Francisco, Manhendra and Francisco 2} and   \cite{1,Laurent Linares and Rosier, Linares Rosier},  respectively. So, our main goal in this paper is to study all these equations in a unified way.

Under the above conditions, the linear operator $\mathcal{A}$ commutes with derivatives and may be seen as a self-adjoint operator on $L^{2}_{p}(\mathbb{T})$ (see Section \ref{preliminares} for notations).
Note also that solutions of the homogeneous equation \eqref{BO} ($f=0$) with initial data $u(0)=u_{0}$ conserve the ``mass'' in the sense that
$$2\pi \widehat{u}(0,t)=2 \pi \widehat{u_{0}}(0),\;\;\text{for all}\;t\in \mathbb{R},$$
where $\widehat{u}$ stands for the Fourier transform of $u$ in the space variable (see \eqref{foutrans}).

Before proceeding let us make clear the problems we are interested in.\\

\noindent
\textbf{\emph{Exact controllability problem:}} Let $s\in \mathbb{R}$ and $T>0$ be given. Let
 $u_{0}$ and $u_{1}$ in  $ H_{p}^{s}(\mathbb{T})$ be given with $\widehat{u_{0}}(0)=\widehat{u_{1}}(0).$  Can one find a control input $f$ such that the unique solution $u$ of the initial-value problem (IVP)
\begin{equation}\label{2D-BO1}
	\begin{cases}
		\partial_{t}u - \partial_{x}\mathcal{A}u
		=f(x,t), \;\;x\in \mathbb{T},\;\;t\in \mathbb{R},\\
		u(x,0)=u_{0}(x) 
	\end{cases}
\end{equation}
is defined until time $T$ and satisfies $u(x,T)=u_{1}(x) \;\;\text{for all}\;\; x\in \mathbb{T}$?\\

\noindent
\textbf{\emph{Asymptotic stabilizability problem:}} Let $s\in \mathbb{R}$ and  $u_{0}\in H_{p}^{s}(\mathbb{T})$ be given.
Can one define a feedback control law $f=Ku$, for some liner operator $K$, such that the resulting closed-loop system
\begin{equation}\label{2D-BO2}
	\begin{cases}
	 \partial_{t}u - \partial_{x}\mathcal{A}u
	=Ku, \;\;x\in \mathbb{T},\;\;t\in \mathbb{R}^{+},\\
	  u(x,0)=u_{0}(x),
	\end{cases}
\end{equation}
is globally well-defined and asymptotically stable  to an equilibrium point as $t\rightarrow +\infty$?\\

In the present manuscript we use the classical Moment method  (see \cite{Russell})  and a  generalization of Ingham's inequality see (see {\cite[Theorem 4.6]{7}} and \cite{8}), to find a practical  criterion  regarding the eigenvalues associated with the operator $\partial_{x}\mathcal{A}$ to determine if equation \eqref{BO} is exactly controllable and exponentially stabilizable. Therefore, we were able to extend the  techniques used by the authors in \cite{1,  Manhendra and Francisco, Rosier Lionel Zhang} to a wide class of linearized dispersive equations on a periodic domain.

\begin{rem}\label{systemsnot}
Generalizing these techniques to linear systems of two or more equations require additional efforts because the mixed dispersive terms present in the equations generally induce a modification of the orthogonal basis we are considering on $L^{2}_{p}(\mathbb{T})\times L^{2}_{p}(\mathbb{T})$ (see for instance \cite[Proposition. 2.2]{Capistrano  y Andressa}). This usually implies  a loss of regularity of the considered controls (see also \cite[Theorem 2.23]{Micu Ortega Rosier and Zhang}).
\end{rem}

As usual in control theory for dispersive models (see \cite{1, 10, 14, Manhendra and Francisco}), in order   to keep the mass  of \eqref{2D-BO1} conserved,  we define a bounded linear  operator $G:H_{p}^{s}(\mathbb{T})\to H_{p}^{s}(\mathbb{T})$ in the following way: let $g$ be a real non-negative  function in $C_p^\infty(\mathbb{T})$ such that
\begin{equation}\label{gcondition}
	2\pi\widehat{g}(0)=\int_{0}^{2\pi}g(x)\;dx=1,
\end{equation}
and assume
$\text{supp} \;g=\omega \subset \mathbb{T},$ where
$\omega=\{x\in \mathbb{T}: g(x)>0 \}$ is an open interval. The operator $G$ is then defined as
\begin{equation}\label{EQ1}
	G(\phi):=g\phi-g\,\langle\phi,g\rangle, \qquad \phi\in H_{p}^{s}(\mathbb{T}),
\end{equation}
where the first product must be understood in the periodic distributional sense and $\langle\cdot,\cdot\rangle$ denotes the pairing between $\mathscr{P}'$ and $\mathscr{P}$ (see notations below). 

The control input $f$ is then chosen to be of the form $f(\cdot,t)=G(h(\cdot,t)),$ $t\in [0,T]$. As a consequence, the function $h \in L^{2}([0,T];H_{p}^{s}(\mathbb{T}))$ is now viewed as the new control function.

\begin{rem}
	Some remarks concerning the operator $G$ are in order.
\begin{enumerate}
	\item It is not difficult to see that $G$ is self-adjoint in $L^{2}_{p}(\mathbb{T})$ (see \cite[Proposition 3.2]{Manhendra and Francisco}). In addition, the authors in \cite[Remark 2.1]{1} and \cite[Lemma 2.20]{Micu Ortega Rosier and Zhang} showed that  for any $s\in \mathbb{R}$ the operator $G$ acting from $L^{2}\left([0,T]; H^{s}_{p}(\mathbb{T})\right)$ into $L^{2}\left([0,T]; H^{s}_{p}(\mathbb{T})\right)$
	is linear and bounded. 
	
	\item When $s\geq0$ we may write
	$$
	G(\phi)(x)=g(x)\left[ \phi(x)-\int_0^{2\pi} \phi(y)g(y)\,dy \right],
	$$
	which is exactly the operator defined, for instance, in   \cite{1, 10, 14, Manhendra and Francisco}.
	
	\item By recalling that for any $\phi\in \mathscr{P}'$ and $g\in \mathscr{P}$ (see \cite[Corollary 3.167]{6})
	$$
	\langle\phi,g\rangle= 2\pi \sum_{k \in \mathbb{Z}}\widehat{\phi}(k)\widehat{g}(-k),
	$$
	in view of \eqref{gcondition}, we obtain
	\[
	\begin{split}
		\widehat{G(\phi)}(0)=\widehat{\phi}\ast \widehat{g}(0)- \widehat{g}(0)\,\langle\phi,g\rangle
		=\sum_{j \in \mathbb{Z}}\widehat{\phi}(j)\widehat{g}(-j)-\widehat{g}(0)\,\langle\phi,g\rangle
		=0,
	\end{split}
	\]
	where the convolution of two sequences of complex numbers $(\alpha_k)_{k\in\mathbb{Z}}$ and $(\beta_k)_{k\in\mathbb{Z}}$ is the sequence $((\alpha\ast\beta)_k)_{k\in\mathbb{Z}}$ defined by
	$$
	(\alpha\ast\beta)_k=\sum_{j\in \mathbb{Z}}\alpha_j\beta_{k-j}.
	$$
	This implies that any solution $u$ of \eqref{2D-BO1} (with $f(x,t)=G(h(x,t))$) conserves the quantity $2\pi \widehat{u}(0,t)$. In particular, if $\widehat{u_0}(0)=0$ then $2\pi \widehat{u}(0,t)=0$, for any $t\in[0,T]$.
\end{enumerate}
\end{rem}

Next, we turn attention to our criteria to obtain the controllability and stabilization of equation \eqref{BO}. As we will see, they directly link these problems  with some specific properties of the eigenvalues and eigenfunctions associated to the operator $\partial_{x}\mathcal{A}$. To derive our first criterion regarding exact controllability,  we assume that $\partial_{x}\mathcal{A}$ has a countable number of eigenvalues that are all simple,  except by a finite number that have finite multiplicity. Specifically, we will assume that the following hypotheses hold:
\vskip.2cm
\begin{itemize}
    \item [$(H1)$] $\partial_{x}\mathcal{A}\psi_{k}=i\lambda_{k} \psi_{k},$ where  $\psi_{k}$ is defined in \eqref{spi} and  $\lambda_{k}=k a(k),$ for all $k\in \mathbb{Z}.$
\end{itemize}
\vskip.2cm
Note we are counting multiplicities, implying that  the eigenvalues in the sequence $\{i \lambda_{k}\}_{k\in \mathbb{Z}}$ are not necessarily distinct. For each $k_{1}\in \mathbb{Z},$ we set $I(k_{1}):=\{k\in \mathbb{Z}: \lambda_{k}=\lambda_{k_{1}}\}$ and $m(k_{1}):=|I(k_{1})|,$
where $|I(k_{1})|$ denotes the number of elements in $I(k_{1}).$
Concerning the quantity  $m(k_{1})$, we assume the following:
\vskip.2cm
\begin{itemize}
    \item [$(H2)$] $m(k_{1})\leq n_{0},$ for some $n_{0}\in \mathbb{N}$ and for all $k_{1}\in \mathbb{Z},$
\end{itemize}
\vskip.2cm
and
\vskip.2cm
\begin{itemize}
    \item[$(H3)$] there exists $k_{1}^{\ast} \in \mathbb{N}$ such that $m(k_{1})=1,$ for all $k_{1}\in \mathbb{Z}$ with $|k_{1}|\geq k_{1}^{\ast}.$ 
\end{itemize}
\vskip.2cm

Assumptions $(H2)$ and $(H3)$ together say that  all eigenvalues $i\lambda_k$ have finite multiplicity. In addition, they are simple eigenvalues for sufficiently large indices.

If we count only distinct eigenvalues, we may obtain a sequence
$\{\lambda_{k}\}_{k \in \mathbb{I}}$, $\mathbb{I}\subseteq\mathbb{Z}$, with the
property that $\lambda_{k_{1}}\neq \lambda_{k_{2}}$, for any $k_{1},k_{2}\in \mathbb{I}$ with $k_{1}\neq k_{2}$. Our main result at this point reads as follows.

\begin{thm}[Criterion I]\label{ControlLa}
Let $s\in \mathbb{R}$ and assume $(H1), \;(H2),$ and $(H3).$ Suppose that
  \begin{equation}\label{gammaseg}
  \begin{split}
  \gamma&:=\inf_{\substack{k,n\in \mathbb{I}\\k\neq n}}|\lambda_{k}-\lambda_{n}|>0
  \end{split}
  \end{equation}
  and
   \begin{equation}\label{gammalinha}
  \begin{split}
\gamma'&:=\underset{S\subset \mathbb{I}}{\sup}\;
\underset{k\neq n}{\underset{k,n \in \mathbb{I}\backslash S}{\inf}}
|\lambda_{k}-\lambda_{n}|>0,	
  \end{split}
  \end{equation}
where $S$ runs over all finite subsets of $\mathbb{I}.$
	Then for any $T> \frac{2\pi}{\gamma'}$ and   for each $u_{0},\; u_{1}\in H_{p}^{s}(\mathbb{T})$ with $\widehat{u_{0}}(0)=\widehat{u_{1}}(0),$ there exists a function $h\in L^{2}([0,T];H_{p}^{s}(\mathbb{T}))$
	such that the unique solution $u$  of the non-homogeneous system
	\begin{equation}\label{introduc2}
\begin{cases}
u\in C([0,T];H^{s}_{p}(\mathbb{T})), \hbox{}\\
   \partial_{t}u(t)= \partial_{x}\mathcal{A}u(t)+ G(h)(t)\in H_{p}^{s-r}(\mathbb{T}),  t\in (0,T),\\
    u(0)=u_{0}\in H_{p}^{s}(\mathbb{T}), 
\end{cases}
\end{equation}
	satisfies $u(T)=u_{1}.$ Furthermore,
	\begin{equation}\label{hboun}
	   \|h\|_{L^{2}([0,T];H_{p}^{s}(\mathbb{T}))} \leq \nu\; (\|u_{0}\|_{H_{p}^{s}(\mathbb{T})}
	+\|u_{1}\|_{H_{p}^{s}(\mathbb{T})}),
	\end{equation}
for some positive constant $\nu \equiv \nu(s,g,T).$
\end{thm}

\begin{rem}\label{partic1} Note that if $\gamma'$   defined in \eqref{gammalinha}  is infinite ($\gamma'=+\infty$), then \eqref{introduc2} is exactly controllable for any positive time $T.$ In particular, if $(H1)$ holds and
    \begin{itemize}
        \item [(i)]  $a(-k)=a(k),$ for all $k\in \mathbb{I}$;

        \item[(ii)] $\displaystyle{\lim_{|k|\rightarrow +\infty}| (k+1) a(k+1)- k a(k)|=+\infty},$ where $k$ runs over $\mathbb{I},$
    \end{itemize}
    then system \eqref{introduc2} is exactly controllable for any $T>0.$ In fact, from  (i) we infer that $\lambda_{k}=-\lambda_{-k}$ for all $k\in \mathbb{I}.$  On the other hand, property (ii) yields that  $\gamma'=+\infty$ and the real sequence $\{\lambda_{k}\}_{k\in \mathbb{I}}$ is strictly increasing/decreasing for $k\in \mathbb{I}$ with $|k|>k_{1}^{\ast}$ for some $k_{1}^{\ast}$, implying that$(H2)$-$(H3)$ hold. Also,
   since terms of the sequence  $\{\lambda_{k}\}_{k\in \mathbb{I}}$ are distinct, it is clear that \eqref{gammaseg} holds.
\end{rem}

   Property (ii) in Remark \ref{partic1} implies the so called \textquotedblleft asymptotic gap condition" for the eigenvalues $\{i \lambda_{k}\}_{k\in \mathbb{I}}$ associated with the operator $\partial_{x}\mathcal{A}.$ This property is crucial to obtain the exact controllability for any  $T>0.$ It appears that many dispersive models hold the properties (i) and (ii). For instance, the linearized KdV equation  \cite{10,14}, the linearized Benjamin-Ono equation  \cite{1}, and the linearized Benjamin equation  \cite{Manhendra and Francisco}. See Figure \ref{fig1} for an illustrative figure.
 \begin{figure}[h!]
	\begin{center}
		\includegraphics[width=15cm]{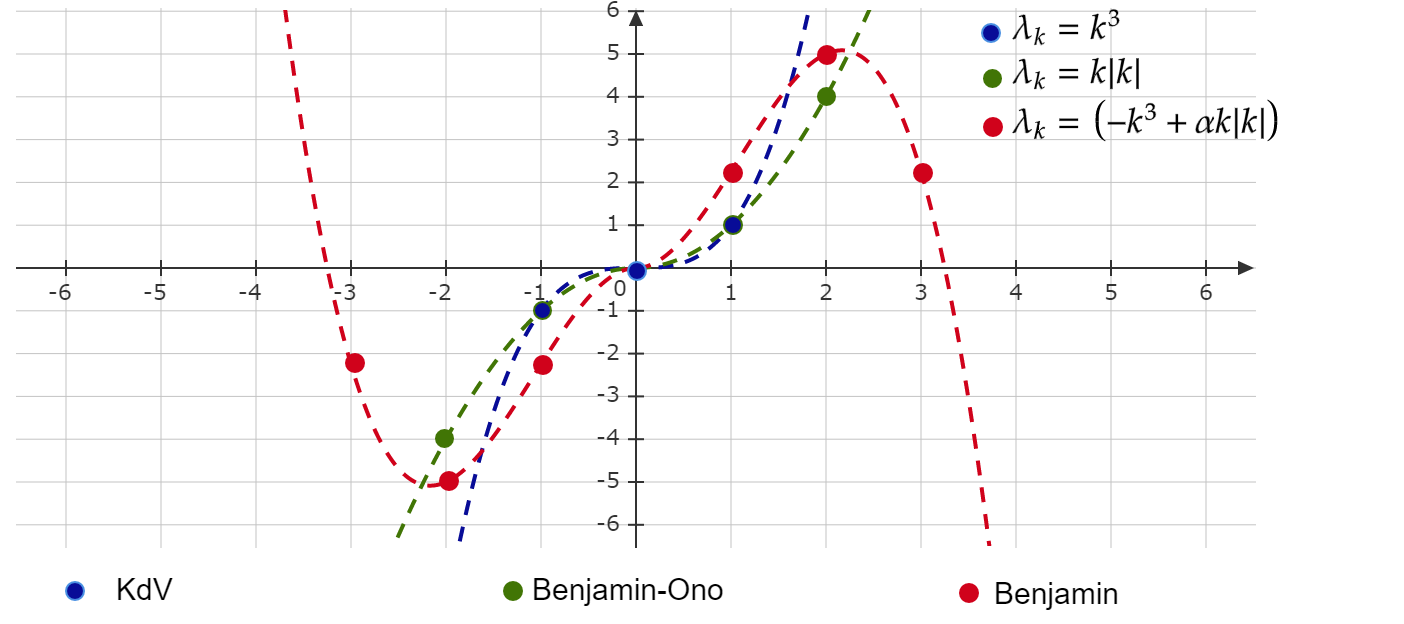}
		\caption{Dispersion of $\lambda_{k}$'s for
		KdV, Benjamin-Ono and Benjamin equations}\label{fig1}
	\end{center}
\end{figure}

Next we shall prove that even when we have an infinity quantity of repeated eigenvalues associated with  $\partial_{x}\mathcal{A}$  in a particular form,  we can still obtain an exact controllability result. This will provide our second criterion. For this, we will assume $(H1)$ and
\vskip.2cm
\begin{itemize}
    \item [$(H4)$] there are $n_{0},k_{1}^{\ast}\in \mathbb{N}$ such that  $m(k_{1})\leq n_{0},$  for all $k_{1}\in \mathbb{Z}$ with $|k_{1}|< k_{1}^{\ast}.$ In addition, $m(k_1)=2$ for all $|k_{1}|\geq k_{1}^{\ast}$.
\end{itemize}
\vskip.2cm
and
\vskip.2cm
\begin{itemize}
    \item [$(H5)$]  $a(-k)=-a(k),$ for all $k\in \mathbb{Z}$ with $|k|\geq k_{1}^{\ast}.$ 
\end{itemize}
\vskip.2cm

Assumption $(H4)$ says that, except near the origin, all eigenvalues are double. Moreover, in view of $(H5)$,  $\lambda_{k}=\lambda_{-k},$ for all $|k|\geq k_{1}^{\ast}.$ This implies that $I(k_1)=\{-k_1,k_1\}$ for $|k_1|\geq k_1^*$.

As before, if we are interested in counting only the distinct eigenvalues we can obtain a set
$$
\mathbb{J}\subset\{-k_1^*+1,-k_1^*+2, \ldots\}
$$
such that the sequence
$\{\lambda_{k}\}_{k \in \mathbb{J}}$ has the property that $\lambda_{k_{1}}\neq \lambda_{k_{2}}$, for any $k_{1},k_{2}\in \mathbb{J}$, with $k_{1}\neq k_{2}.$

Our second result regarding controllability reads as follows.
\begin{thm}[Criterion II]\label{ControlLag}
Let $s\in \mathbb{R}$ and assume $(H1), \;(H4),$ and $(H5).$ Suppose
  \begin{equation}\label{gammasegg}
  \begin{split}
  \tilde{\gamma} &:=\inf_{\substack{k,n\in \mathbb{J}\\k\neq n}}|\lambda_{k}-\lambda_{n}|>0
  \end{split}
  \end{equation}
  and
   \begin{equation}\label{gammalinhag}
  \begin{split}
\tilde{\gamma}'&:=\underset{S\subset \mathbb{J}}{\sup}\;
\underset{k\neq n}{\underset{k,n \in \mathbb{J}\backslash S}{\inf}}
|\lambda_{k}-\lambda_{n}|>0,	
  \end{split}
  \end{equation}
where $S$ runs over the finite subsets of $\mathbb{J}.$
	Then for any $T> \frac{2\pi}{\tilde{\gamma}'}$ and   for each $u_{0},\; u_{1}\in H_{p}^{s}(\mathbb{T})$ with $\widehat{u_{0}}(0)=\widehat{u_{1}}(0),$ there exists a function $h\in L^{2}([0,T];H_{p}^{s}(\mathbb{T}))$
	such that the unique solution $u$ of the non-homogeneous system
	\eqref{introduc2} satisfies $u(T)=u_{1}.$ Moreover, there exists a positive constant $\nu \equiv \nu(s,g,T)$
	such that \eqref{hboun} holds.
\end{thm}

\begin{rem}\label{partic}
    If hypotheses $(H1),$ $(H4)$  and $(H5)$ hold with $$ \displaystyle{\lim_{k\rightarrow +\infty}|(k+1) a(k+1)- k a(k)|=+\infty},$$ where $k$ take values in $\mathbb{J},$
    then the system \eqref{introduc2} is exactly controllable for any $T>0.$  
\end{rem}

It is not difficult to see that the linear  Schr\"odinger equation holds the assumptions in Theorem \ref{ControlLag} and Remark\ref{partic}. See Figure \ref{fig2} for an illustration of the eigenvalues. Actually, the exact controllability and exponential stabilization for the linear (and  nonlinear cubic) Schr\"odinger equations were proved in   \cite{ Rosier Lionel Zhang}, where the authors used $f(x,t)=Gh(x,t):=g(x)h(x,t)$ as a control input. Here we show that the control input as described in \eqref{EQ1} also serves to prove the exact controllability. The advantage of using this control input is that it allow us to get a controllability  and stabilization result for the linear Schr\"odinger equation in the Sobolev space $H^{s}_{p}(\mathbb{T})$ for any $s\in \mathbb{R}$.
\begin{figure}[h!]
	\begin{center}
		\includegraphics[width=15cm]{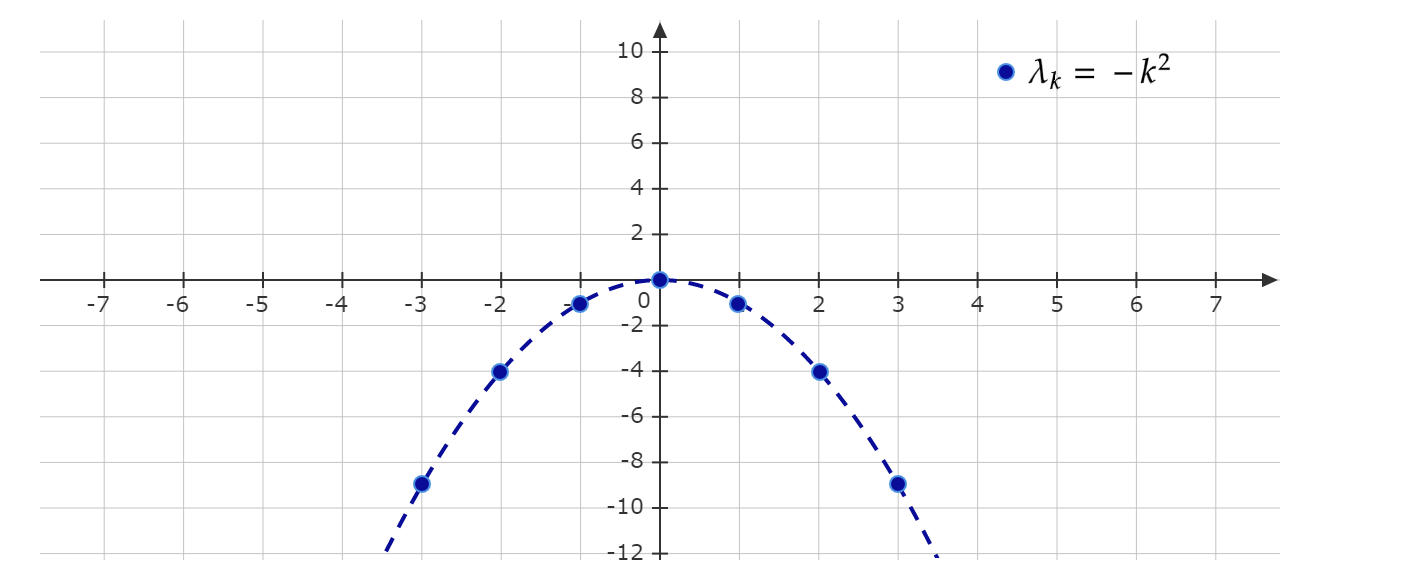}
		\caption{Dispersion of $\lambda_{k}$'s for
			Schr\"odinger equation}\label{fig2}
	\end{center}
\end{figure}

Attention is now turned to our stabilization results. In what follows, $G^\ast$ denotes the adjoint operator of $G$. We will prove that
if one chooses the feedback law $Ku=-GG^{\ast}u$ then the closed-loop system \eqref{2D-BO2} is exponentially stable. More precisely, we have the following.

\begin{thm}\label{st351}
	Let  $g$ be as in \eqref{gcondition}  and let  $s\in \mathbb{R}$ be given. Under the assumptions of Theorem \ref{ControlLa}  or  Theorem \ref{ControlLag},  there exist positive constans $M=M( g, s)$ and $\alpha=\alpha(g)$ such that for any
	$u_{0}\in H_{p}^{s}(\mathbb{T})$  the unique solution $u$
	of the closed-loop system
	\begin{equation*}%\label{estag}
\begin{cases}
u\in C([0,+\infty);H^{s}_{p}(\mathbb{T})), \hbox{}\\
   \partial_{t}u(t)= \partial_{x}\mathcal{A}u(t)- GG^{\ast}u(t)\in H_{p}^{s-r}(\mathbb{T}),  t> 0,\\
    u(0)=u_{0}\in H_{p}^{s}(\mathbb{T}), 
\end{cases}
\end{equation*}	
satisfies
$$\|u(\cdot,t)-\widehat{u_{0}}(0)\|_{H_{p}^{s}(\mathbb{T})}\leq M
	e^{-\alpha t}\|u_{0}-\widehat{u_{0}}(0)\|_{H_{p}^{s}(\mathbb{T})},\;\;\;\text{for all}\;\;
	t\geq 0.$$
\end{thm}

The feedback law $Ku=-GG^{\ast}u$ in Theorem \ref{st351} is the simplest one providing  the exponential decay with a fixed exponential rate. However, by changing the feedback law one is able to show that  the resulting closed-loop system actually has an arbitrary exponential decay rate. More precisely,

\begin{thm}\label{estabilization}
	Let $s\in \mathbb{R},$  $\lambda>0,$ and $u_{0}\in H_{p}^{s}(\mathbb{T})$ be given.  Under the assumptions of Theorem \ref{ControlLa} or Theorem \ref{ControlLag}, there exists a bounded linear  operator $K_{\lambda}$  from $H_{p}^{s}(\mathbb{T})$ to $H_{p}^{s}(\mathbb{T})$ such that
	the unique solution $u$ of the closed-loop system
	\begin{equation}\label{estag}
\begin{cases}
u\in C([0,+\infty);H^{s}_{p}(\mathbb{T})), \hbox{}\\
   \partial_{t}u(t)= \partial_{x}\mathcal{A}u(t)+ K_{\lambda}u(t)\in H_{p}^{s-r}(\mathbb{T}),  t> 0,\\
    u(0)=u_{0}\in H_{p}^{s}(\mathbb{T}), 
\end{cases}
\end{equation}		
satisfies
	$$\|u(\cdot,t)-\widehat{u_{0}}(0)\|_{H_{p}^{s}(\mathbb{T})}\leq
	M\;e^{-\lambda\;t}\|u_{0}-\widehat{u_{0}}(0)\|_{H_{p}^{s}(\mathbb{T})},$$
for all $t\geq0,$ and some positive constant $M=M(g,\lambda, s).$
\end{thm}

 The paper is organized  as follows:
In section \ref{preliminares}  a series of preliminary results that will be used throughout this work are recalled.
 In Section \ref{section4}  we prove well-posedness results. The main  results regarding controllability and    stabilization are proved  in Sections \ref{section5} and  \ref{section6}, respectively. In Section \ref{section6g}, we apply our general criteria to establish the corresponding results regarding exact controllability and exponential stabilization for the linearized Smith equation,  the linearized dispersion-generalized Benjamin-Ono equation, the fourth-order Schr\"odinger and a higher-order Schr\"odinger equation.  Finally, in Section \ref{conc-rem} some concluding remarks and future works are presented.

\section{Preliminaries}\label{preliminares}
In this section we introduce some basic notations and recall the main tools to obtain our results. We denote by $\mathscr{P}$ the space $C^\infty_p(\mathbb{T})$  of all $C^\infty$  functions that are $2\pi$-periodic. By $\mathscr{P}'$ (the dual of $\mathscr{P}$) we denote the space of all periodic distributions.  By $L^2_p(\mathbb{T})$ we denote the standard space of the square integrable $2\pi$-periodic functions.
It is well-known that the sequence  $\{ \psi_{k} \}_{k\in \mathbb{Z}}$  given by
\begin{equation}\label{spi}
\psi_{k}(x):=\frac{e^{ikx}}{\sqrt{2 \pi}},\;\; k\in \mathbb{Z},\;\;x\in \mathbb{T}
\end{equation}
is an orthonormal basis  for $L_{p}^{2}(\mathbb{T})$. The Fourier transform of $v\in \mathscr{P}'$ is defined as
\begin{equation}\label{foutrans}
	\widehat{v}(k)=\frac{1}{2 \pi}\langle v,e^{-ikx}\rangle,\;\; \;k \in \mathbb{Z}.
\end{equation}

Next we introduce the periodic Sobolev spaces. For a more detailed description and properties of
these spaces, we refer the reader to \cite{6}. Given $s\in\Real$, the (periodic) Sobolev space of order $s$ is defined as
$$H^{s}_{p}(\mathbb{T})=\left\{ v\in \mathscr{P}' \left| \right. \|v\|_{H^{s}_{p}(\mathbb{T})}^{2}:=2\pi\sum_{k=-\infty}^{\infty}
(1+|k|)^{2s}|\widehat{v}(k)|^{2}<\infty \right\}.$$
We 
consider the space $H^{s}_{p}(\mathbb{T}) $ as a  Hilbert space endowed
with the inner product
\begin{equation}\label{innerprod}
\displaystyle{(h\,, \,v)_{H^{s}_{p}(\mathbb{T})}= 2\pi \sum_{k\in \mathbb{Z}}
(1+|k|)^{2s}\widehat{h}(k)\;\overline{\widehat{v}(k)},}
\end{equation}
For any $s\in \Real$, $(H^{s}_{p}(\mathbb{T}))'$, the topological dual of $H^{s}_{p}(\mathbb{T})$, is isometrically isomorphic to $H^{-s}_{p}(\mathbb{T})$, where the duality is implemented by the pairing
$$
\displaystyle{\langle h, v\rangle_{H^{-s}_{p}(\mathbb{T})\times H_{p}^{s}(\mathbb{T})}= 2\pi \sum_{k\in \mathbb{Z}}
\widehat{h}(k)\;\overline{\widehat{v}(k)}},\;\;\text{for all}\; v \in H^{s}_{p}(\mathbb{T}),\;h\in H^{-s}_{p}(\mathbb{T}).
$$

\begin{rem}
	It is well-known that any distribution $v\in \mathscr{P}'$ may be written as (see, for instance, \cite[page 188]{6})
	\begin{equation}\label{prep}
		v=\sqrt{2\pi}\sum_{k\in\mathbb{Z}}\widehat{v}(k)\psi_{k},
	\end{equation}
where the series converges in the sense of $\mathscr{P}'$. In particular, any $v\in H^{s}_{p}(\mathbb{T})$, $s\in\Real$, can be written in the form \eqref{prep}.
\end{rem}

 We also consider the closed subspace
 $$H^{s}_{0}(\mathbb{T}):=\left\{ v\in H^{s}_{p}(\mathbb{T)}\left|\right. \;\; \widehat{v}(0)=0\right\}.$$
It can be seen that  if $s_{1}, \;s_{2}\in \mathbb{R}$ with $s_{1}\geq s_{2}$ then
$H_{0}^{s_{1}}(\mathbb{T}) \hookrightarrow H_{0}^{s_{2}}(\mathbb{T}),$ where the  embedding is dense. We denote $H_{0}^{0}(\mathbb{T})$ by $L_{0}^{2}(\mathbb{T}).$
 In particular, $L_{0}^{2}(\mathbb{T})$ is a closed subspace of
	$L^{2}_{p}(\mathbb{T})$.

We continue with some characterization of Riesz basis in Hilbert spaces (see \cite{9} for more details). In what follows, $J$ represents a countable set of indices which could be finite or infinite.

\begin{thm}\label{BasRieszTheo}
	Let $\{x_{n}\}_{n\in J}$ be a sequence in a Hilbert space $H.$ Then the following statements are equivalent.
	
	\begin{enumerate}
		\item $\{x_{n}\}_{n\in J}$ is a Riesz basis for $H$ (see \cite[Definition 7.9]{9}).

		\item $\{x_{n}\}_{n\in J}$ is complete in $H$ (see  {\cite[Definition 1.25]{9}}) and there exist constants $A,\;B>0$ such that
		$$\text{for all}\;\;c_{1},...,c_{N}\;\;\text{scalars,}\;\;A\sum_{n=1}^{N}|c_{n}|^{2}\leq \|\sum_{n=1}^{N}c_{n}x_{n}\|^{2}_{H}
		\leq B\sum_{n=1}^{N}|c_{n}|^{2}.$$
		
		\item There is an equivalent inner product $(\cdot, \cdot)$ for $H$ such that  $\{x_{n}\}_{n\in J}$ is an orthonormal basis for $H$ with respect to $(\cdot, \cdot).$
		
		\item $\{x_{n}\}_{n\in J}$ is a complete Bessel sequence (see {\cite[Definition 7.1]{9}}) and possesses a biorthogonal system $\{y_{n}\}_{n\in J}$  (see {\cite[Definition 4.10]{9}}) that is also
		a complete Bessel sequence.
	\end{enumerate}
\end{thm}
\begin{proof}
	See {\cite[Theorem 7.13]{9}}.
\end{proof}

Finally,  we recall the generalized Ingham's inequality.

\begin{thm}\label{InghamG2}
	Let $\{\lambda_{k}\}_{k\in J}$ be a family of real numbers, satisfying the uniform gap condition
	$$\gamma=	\underset{k\neq n}{\underset{k,n \in J}{\inf}} |\lambda_{k}-\lambda_{n}|>0.$$
	Set
	$$\gamma'=\underset{S\subset J}{\sup}\;
	\underset{k\neq n}{\underset{k,n \in J\backslash S}{\inf}}
	|\lambda_{k}-\lambda_{n}|>0,$$	
	where $S$ runs over all finite subsets of $J.$
	
	If $I$ is a bounded interval of length $|I|> \frac{2\pi}{\gamma'},$ then there exist positive constants $A$ and $B$ such that
	$$A\sum_{k \in J}|c_{k}|^{2}\leq \int_{I}|f(t)|^{2}dt\leq B\sum_{k \in J}|c_{k}|^{2},$$
	for all functions of the form $f(t)=\sum\limits_{k \in J}c_{k}e^{i\lambda_{k}t}$ with  square-summable complex coefficients $c_{k}.$
\end{thm}
\begin{proof}
	See {\cite[page 67]{7}}.
\end{proof}

For further generalizations of the Ingham inequality  (see \cite{8}) we refer the reader to \cite{Ball and Slemrod} and  \cite{7}.

\section{Well-posedness}
\label{section4}
In this section we establish a global well-posedness result for system  \eqref{introduc2}. We start with some results concerning the homogeneous equation. This results are quite standard but for the sake of completeness we bring the main steps.

\begin{prop}\label{OGU}
 Let $r$ be as in \eqref{limite}. For any $u_{0}\in H_{p}^{r}(\mathbb{T})$, the   homogeneous problem
\begin{equation}\label{introduc}
\begin{cases}
u\in C(\mathbb{R};H^{r}_{p}(\mathbb{T}))\cap C^{1}(\mathbb{R},L^{2}_{p}(\mathbb{T})), \\
   \partial_{t}u = \partial_{x}\mathcal{A}u \in L_{p}^{2}(\mathbb{T}),\quad  t\in \mathbb{R},\\
    u(0)=u_{0},  
\end{cases}
\end{equation}
has a unique solution.
\end{prop}
\begin{proof}
First note that from Plancherel's identity, for any $\varphi, \psi \in  D(\partial_{x}\mathcal{A})=H^{r}_{p}(\mathbb{T}) $,
we have
\begin{align*}
( \partial_{x}\mathcal{A}\varphi, \psi)_{L^{2}_{p}(\mathbb{T})}
&= 2\pi \sum_{k=-\infty}^{+\infty}\widehat{\partial_{x}\mathcal{A}\varphi}(k) \overline{\widehat{\psi}(k)}\\
&= 2\pi \sum_{k=-\infty}^{+\infty} i k a(k)\widehat{\varphi}(k) \overline{\widehat{\psi}(k)} \\
&=-2\pi \sum_{k=-\infty}^{+\infty} \widehat{\varphi}(k) \overline{ (-i k) a(k) \widehat{\psi}(k)}\\
&=-(\varphi, \partial_{x}\mathcal{A}\psi)_{L^{2}(\mathbb{T})},
\end{align*}
which implies that $\partial_{x}\mathcal{A}$   is skew-adjoint. Hence, Stone's theorem gives that  $\partial_{x}\mathcal{A}$ generates a strongly continuous unitary group $\{U(t)\}_{t\in \mathbb{R}}$
on $L^{2}_{p}(\mathbb{T}).$ 
Therefore, Theorem 3.2.3 in  \cite{3} yields the desired result.
\end{proof}

Proposition \ref{OGU} provides the well-posedness theory for \eqref{introduc} only for initial data in $H_{p}^{r}(\mathbb{T})$. However, we can still obtain the well-posedness for initial data  in $H_{p}^{s}(\mathbb{T})$ for any $s\in \mathbb{R}.$ To do so, one needs a more accurate description of the unitary group $\{U(t)\}_{t\in \mathbb{R}}$. At least in a formal level, by taking Fourier's transform in the spatial variable, it is not difficult to see that  the solution of \eqref{introduc} may be written as 
\begin{equation}\label{prop2}
\widehat{u}(t)(k)=e^{ik a(k) t}\widehat{u_{0}}(k),\;\; k\in \mathbb{Z},
\end{equation}
or, by taking the inverse Fourier transform,
\begin{equation}\label{solf}
 u(t)=\left(e^{ik a(k) t}\widehat{u_{0}}(k)\right)^{\vee}, \;\;\;t\in \mathbb{R}.
\end{equation}
This means that
\begin{equation}\label{solf2}
 u(x,t)=\sum_{k\in \mathbb{Z}}
e^{ik a(k) t}\widehat{u_{0}}(k)e^{ikx},\;\;\;t\in \mathbb{R},
\end{equation}
must be the unique solution of  \eqref{introduc}.

The above calculation suggests that, in a rigorous way, we may define the family of linear operators
\;$U:\mathbb{R}\rightarrow \mathcal{L}(H^{s}_{p}(\mathbb{T}))$ by
\begin{equation}\label{semi2}
\begin{split}
 t\rightarrow U(t)\varphi:=e^{\partial_{x}\mathcal{A}t}\varphi
 &=(e^{ik a(k) t}\widehat{\varphi}(k))^{\vee},
\end{split}
\end{equation}
in such a way that the solution  of \eqref{introduc} now becomes $u(t)=U(t)u_{0},\;t\in \mathbb{R}.$

From the growth condition \eqref{limite} and classical results on the semigroup theory (see for instance  \cite{3},   \cite{4} or \cite{6} for additional details), we can show that the family of operators $\{U(t)\}_{t\in\mathbb{R}}$ given by \eqref{semi2} indeed
defines a strongly continuous one-parameter unitary  group  on $H^{s}_{p}(\mathbb{T})$, for any $s\in\mathbb{R}$.  Additionally,  if $u(t)=U(t)u_{0}$ with $u_{0}\in H^{s}_{p}(\mathbb{T}),$ then
\begin{equation*}%\label{lim2}
    \lim_{h\rightarrow 0}\left\|\frac{u(t+h)-u(t)}{h}-\partial_{x}\mathcal{A}u \right\|_{H_{p}^{s-r}(\mathbb{T})}=0,
\end{equation*}
uniformly with respect $t\in\mathbb{R}.$  In particular, the following result holds.

\begin{thm}\label{EU1}
	Let $s\in\mathbb{R}$ and $u_{0}\in H_{p}^{s}(\mathbb{T})$ be given. Then the homogeneous problem
	\begin{equation*}%\label{introducag}
\begin{cases}
u\in C(\mathbb{R};H^{s}_{p}(\mathbb{T})), \\
   \partial_{t}u = \partial_{x}\mathcal{A}u \in H_{p}^{s-r}(\mathbb{T}), t\in \mathbb{R},\\
    u(0)=u_{0},  
\end{cases}
\end{equation*}
has a unique solution.
\end{thm}

Next, we deal with the well-posedness of the non-homogenous linear problem \eqref{introduc2}.
\begin{lem}\label{EUNH}
Let $0< T<\infty,$ $s\in \mathbb{R},$  $u_{0}\in H_{p}^{s}(\mathbb{T}),$ and $h\in L^{2}([0,T];H_{p}^{s}(\mathbb{T})).$
Then, there exists a unique mild solution
$u\in C([0,T],H_{p}^{s}(\mathbb{T})) $ for the IVP \eqref{introduc2}.
\end{lem}
\begin{proof} This is a consequence of Corollary 2.2 and  Definition 2.3 in \cite[page 106]{4}, and the fact that
$G(h)\in L^{1}([0,T];H_{p}^{s}(\mathbb{T})).$ Furthermore, the unique (mild) solution of \eqref{introduc2} is given by
\begin{equation}\label{inteeq}
	u(t)=U(t)u_{0}+\int_{0}^{t}U(t-t')Gh(t')dt', \qquad t\in[0,T].
\end{equation}
This completes the proof of the lemma.
\end{proof}

\section{Proof of the Control Results}
\label{section5}
In this section we   use the classical moment method (see \cite{Russell}) to show the criteria I and II regarding exact controllability for \eqref{introduc2}. First of all, by replacing $u_1$ by $u_1-U(T)u_0$ if necessary, we may assume without loss of generality that $u_{0}=0$ (see \cite[page 10]{Manhendra and Francisco}), implying that $\widehat{u_{1}}(0)=\widehat{u_{0}}(0)=0.$ Consequently, if we write  $u_{1}(x)=\sum\limits_{k\in \mathbb{Z}}c_{k}\;\psi_{k}(x)$ with  $\psi_{k}$ as  in \eqref{spi} then $c_0=0$.

Our first result is a characterization to get the exact controllability for \eqref{introduc2}. Its proof is similar to the proof of Lemma 4.1 in \cite{Manhendra and Francisco},  passing to the frequency space when necessary; so we omit the details.

\begin{lem}\label{caractc}
Let $ s\in \mathbb{R}$ and $T>0$ be given. Assume $u_{1}\in H^{s}_{p}(\mathbb{T})$ with $\widehat{u_{1}}(0)=0.$
Then, there exists
$h\in L^{2}([0,T],H^{s}_{p}(\mathbb{T}))$ such that the solution of the IVP \eqref{introduc2} with initial data $u_{0}=0$ satisfies $u(T)=u_{1}$ if and only if
\begin{equation}\label{CEQ}
    \int_{0}^{T}\left\langle Gh(\cdot,t),\varphi(\cdot,t)\right\rangle_{H^{s}_{p}\times(H^{s}_{p})'}dt
  =\left\langle u_{1},\varphi_{0}\right\rangle_{H^{s}_{p}\times(H^{s}_{p})'},
\end{equation}
for any $\varphi_{0}\in (H^{s}_{p}(\mathbb{T}))'$, and $\varphi$
is the solution of the adjoint system
\begin{equation}\label{adsis}
\begin{cases}
\varphi\in C([0,T]:\left(H_{p}^{s}(\mathbb{T})\right)'),\\
    \partial_{t}\varphi=\partial_{x}\mathcal{A}\varphi \in H_{p}^{-s-r}(\mathbb{T}), \quad t>0,\\
    \varphi(T)=\varphi_{0}. 
\end{cases}
\end{equation}
\end{lem}

Next corollary is a consequence of Lemma \ref{caractc}. Having in mind its importance, we write the proof.

 \begin{cor} \label{controloperator1}
  Let $s\in \mathbb{R},$ $T> 0,$ and $u_{1}\in H^{s}_{p}(\mathbb{T})$ with $\widehat{u_{1}}(0)=0$ be given. Then, there exists $h\in L^{2}([0,T];H_{p}^{s}(\mathbb{T})),$ such that the unique solution of the IVP \eqref{introduc2} with initial data $u_{0}=0$ satisfies $u(T)=u_{1}$ if and only if there exists $\delta>0$ such that
  \begin{equation}\label{ob1}
   \int_{0}^{T}\|G^{\ast}U(\tau)^{\ast}\phi^{\ast}\|
    ^{2}_{(H_{p}^{s}(\mathbb{T}))'}(\tau) \; d\tau \geq \delta^{2}
    \|\phi^{\ast}\|^{2}_{(H^{s}_{p}(\mathbb{T}))'},
  \end{equation}
for any $\phi^{\ast} \in (H^{s}_{p}(\mathbb{T}))'.$
 \end{cor}
\begin{proof}$(\Rightarrow)$
	Let $T>0$ and define the linear map $F_{T}:L^{2}([0,T]; H^{s}_{p}(\mathbb{T}))\rightarrow H^{s}_{p}(\mathbb{T})$ by $F_{T}(h)=u(T)$,
	where $u$ is the (mild) solution of \eqref{introduc2} with $u(0)=0.$ From the hypothesis, the map $F_{T}$ is onto and, given $u_1\in H^{s}_{p}(\mathbb{T})$,
	\begin{equation}\label{cont1}
F_T(h)=	u_{1}=\int_{0}^{T}U(T-s)(G(h))(s)\;ds,
	\end{equation}
for some $h\in L^{2}([0,T]; H^{s}_{p}(\mathbb{T}))$. Therefore, 
$$
{\footnotesize
\begin{split}
	\|F_{T}(h)\|_{H^{s}_{p}(\mathbb{T})}
	&\leq \int\limits_{0}^{T}\left\|U(T-s)(G(h))(s)
	\right\|_{H^{s}_{p}(\mathbb{T})}\;ds
	\leq c\int\limits_{0}^{T}\|h\|_{H^{s}_{p}(\mathbb{T})}\;ds
	\leq cT^{\frac{1}{2}} \|h\|_{L^{2}([0,T];H^{s}_{p}(\mathbb{T}))},
\end{split}}
$$
for some constant $c$ depending on $g$. So, $F_{T}$ is a bounded linear operator. Thus, $F_{T}^{\ast}$ exists, is a bounded linear operator, and it is one-to-one (see  Rudin  \cite[Corollary b) page 99]{Rudin}).
	Also, from Theorem 4.13  in \cite{Rudin} (see also  \cite[page 35]{Coron}), we have that there exists
	$\delta >0$ such that
	\begin{equation}\label{caractc12}
	\left\|F_{T}^{\ast}(\phi^{\ast})
	\right\|_{\left(L^{2}([0,T];H^{s}_{p}(\mathbb{T}))\right)'}
	\geq \delta \;\|\phi^{\ast}\|_{\left(H^{s}_{p}(\mathbb{T})\right)'},
	\;\;\;\text{ for all}\;\;
	\phi^{\ast}\in \left(H^{s}_{p}(\mathbb{T})\right)'.
	\end{equation}

	From Lemma \ref{caractc},  we have that the solution $u$ of  \eqref{introduc2} with $u_{0}=0$ satisfies
	\begin{equation}\label{caractc6}
	\int_{0}^{T}\left\langle Gh(\cdot,t),\varphi(\cdot,t)\right\rangle_{H^{s}_{p} \times(H^{s}_{p})'}dt
	-\left\langle u_{1},\varphi_{0}\right\rangle_{H^{s}_{p}\times(H^{s}_{p})'}=0,
	\end{equation}
	for any $\varphi_{0}\in (H^{s}_{p}(\mathbb{T}))',$ and $\varphi$
	the solution of  the adjoint system \eqref{adsis}.
By noting that $\varphi(\cdot,t)=U(T-t)^{\ast}\varphi_{0}$,  it follows from
	\eqref{caractc6} that
$${\small
	\begin{split}
	\int\limits_{0}^{T}\left\langle h(\cdot,t),G^{\ast}U(T-t)^{\ast}\varphi_{0}
	\right\rangle_{H^{s}_{p}(\mathbb{T})\times
		(H^{s}_{p}(\mathbb{T}))'}dt
	&=\left\langle u(T),\varphi_{0}\right
	\rangle_{H^{s}_{p}(\mathbb{T})\times
		(H^{s}_{p}(\mathbb{T}))'}\\
	&=\left\langle F_{T}(h),\varphi_{0}\right
	\rangle_{H^{s}_{p}(\mathbb{T})\times
		(H^{s}_{p}(\mathbb{T}))'}\\
	&=\left\langle h\;,\;F_{T}^{\ast}\varphi_{0}\right\rangle
	_{L^{2}([0,T];H^{s}_{p}(\mathbb{T})) \times \left(L^{2}([0,T];H_{p}^{s}(\mathbb{T}))\right)'}.
	\end{split}}
$$
Identifying $L^{2}([0,T];H_{p}^{s}(\mathbb{T}))$ with its dual  one infers $F_{T}^{\ast}=G^{\ast}U(T-t)^{\ast},$ and using \eqref{caractc12},
	we have
$$\left\|G^{\ast}U(T-t)^{\ast}(\phi^{\ast})
	\right\|_{L^{2}([0,T];(H^{s}_{p}(\mathbb{T}))')}
	\geq \delta \;\|\phi^{\ast}\|_{\left(H^{s}_{p}(\mathbb{T})\right)'},
	\;\;\;\text{ for all}\;\;
	\phi^{\ast}\in \left(H^{s}_{p}(\mathbb{T})\right)',$$
or, equivalently,
$$\int_{0}^{T}\|G^{\ast}U(T-t)^{\ast}(\phi^{\ast}(x))
	\|^{2}_{(H^{s}_{p}(\mathbb{T}))'}\;dt
	\geq \delta^{2}\; \|\phi^{\ast}\|^{2}_{(H_{p}^{s}(\mathbb{T}))'},\;\;\;\text{ for all}\;\;
	\phi^{\ast}\in (H^{s}_{p}(\mathbb{T}))'.$$
The change of variables $ \tau=T-t$ yields \eqref{ob1}.

$(\Leftarrow)$ If \eqref{ob1} holds, then
$F_{T}^{\ast}=G^{\ast}U(T-t)^{\ast}$ is onto. It is easy to prove that $F_{T}^{\ast}$ is bounded from $(H_{p}^{s}(\mathbb{T}))'$ into $(L^{2}([0,T];H_{p}^{s}(\mathbb{T})))'.$ Therefore, $F_{T}$ is onto. From computations similar to those above we obtain that \eqref{caractc6} holds. Then Lemma \ref{caractc} implies the result and we conclude the proof of the corollary.
\end{proof}

The following characterization is fundamental to prove the existence of control for \eqref{introduc2} with  initial data $u_{0}=0.$ It provides a method to find the control function $h$ explicitly.

\begin{lem}[Moment Equation]\label{coef}
Let  $ s\in \mathbb{R}$ and $T>0$  be given.
If $$u_{1}(x)=\sum_{l\in \mathbb{Z}}
c_{l}\psi_{l}(x)\;\;\in H^{s}_{p}(\mathbb{T}),$$ is a function such that $\widehat{u_{1}}(0)=0,$
then  the solution $u$ of \eqref{introduc2} with  initial data $u_{0}=0$
satisfies $u(T)=u_{1}$
if an only if there exists $h\in L^{2}([0,T];H^{s}_{p}(\mathbb{T}))$ and
\begin{equation}\label{caract5}
\int_{0}^{T}\left(Gh(x,t),\;
\;e^{-i \lambda_{k} (T-t)} \psi_{k}(x)\right)_{L^{2}_{p}(\mathbb{T})}dt=c_{k},\;\forall\;k\in \mathbb{Z},
\end{equation}
where $\lambda_{k}:=k a(k).$

\end{lem}

\begin{proof}
$(\Rightarrow)$ By taking  $\varphi_{0}=\psi_{k}\in (H^{s}_{p}(\mathbb{T}))'$ in \eqref{adsis},  identity \eqref{solf} implies that
\begin{equation*}%\label{sol1}
\begin{split}
\varphi(x,t)
=\left(e^{-i\lambda_l (T-t)}\widehat{\psi_k}(l)\right)^{\vee}
=\sum_{l\in \mathbb{Z}}
e^{-i\lambda_{l}(T-t)}\widehat{\psi_k}(l)e^{ilx}=e^{-i\lambda_{k}(T-t)}\psi_{k}(x),
\end{split}
\end{equation*}
where in the last identity we used that $\widehat{\psi_{k}}(l)=\frac{1}{\sqrt{2\pi}}\delta_{kl}$, with $\delta_{kl}$ being the Kronecker delta.
Now, using \eqref{CEQ} one gets
\begin{align*}
 \int_{0}^{T}\left(Gh(x,t),\;\varphi(x,t)\right)_{L^{2}_{p}(\mathbb{T})}dt&
 -
 \left(\sum_{l\in \mathbb{Z}}
c_{l}\psi_{l}(x),\;\varphi_{0}(x)\right)_{L^{2}_{p}(\mathbb{T})}=0.
\end{align*}
Therefore, for any $k\in \mathbb{Z}$,
\begin{align*}
 \int_{0}^{T}\left(Gh(x,t),\;e^{-i\lambda_{k}(T-t)}\psi_{k}(x)\right)_{L^{2}_{p}(\mathbb{T})}dt&=
 2\pi  \sum_{j\in \mathbb{Z}} \left(\sum_{l\in \mathbb{Z}}
c_{l}\psi_{l}(x)\right)^{\wedge}(j)\;\overline{\widehat{\psi_{k}}(j)} \\
&= 2\pi \sum_{l\in \mathbb{Z}}
c_{l}
\widehat{\psi_{l}}(k)\;\frac{1}{\sqrt{2\pi}}\\
 &=c_{k},
\end{align*}
as required.\\

\noindent
$(\Leftarrow)$
Now, suppose that there exists $h\in L^{2}([0,T];H^{s}_{p}(\mathbb{T}))$ such that
\eqref{caract5} holds. With similar calculations as above, we  obtain
\begin{equation}\label{inte4}
\int_{0}^{T}\left(Gh(x,t),\;
\;e^{-i\lambda_{k}(T-t)}\;\psi_{k}(x)\right)_{L^{2}_{p}(\mathbb{T})}dt-\left( u_{1}(x),\;\psi_{k}(x)\right)_{L^{2}_{p}(\mathbb{T})}=0,\;
\;k\in \mathbb{Z}.
\end{equation}

 For any  $\varphi_{0}\in C_{p}^{\infty}(\mathbb{T})$ we may write
 $$\varphi_{0}(x)=\sum_{k\in \mathbb{Z}}\sqrt{2\pi}\widehat{\varphi_{0}}(k)\;\psi_{k}(x),$$
 where the series converges uniformly. Thus, using the properties of the inner product and \eqref{inte4}, we get
\begin{equation}\label{inte5}
\int_{0}^{T}\left(Gh(x,t),\;
\varphi(x,t)\right)_{L^{2}_{p}(\mathbb{T})}dt=
\left( u_{1}(x),\;\varphi_{0}(x)\right)_{L^{2}_{p}(\mathbb{T})},
\end{equation}
where we used that the solution of \eqref{adsis} may be expressed as
\begin{align*}
\varphi(x,t)&=\sum_{k\in \mathbb{Z}}
e^{-i\lambda_{k}(T-t)}\widehat{\varphi_{0}}(k)e^{ikx}
\end{align*}
with the series converging uniformly. By density, \eqref{inte5} holds for any $\varphi_{0}\in (H^{s}_{p}(\mathbb{T}))'$. An application of Lemma \ref{caractc} then gives the desired result.
\end{proof}

\begin{lem}\label{invertmatrix}
For $ \psi_{k} $ as in \eqref{spi} and $G$ as in \eqref{EQ1}, define
\begin{equation}\label{invertmatrix1}
m_{j,k}:=\widehat{G(e^{ijx})}(k)=\int_{0}^{2\pi}G(\psi_{j})(x) \overline{\psi_{k}}(x)\;dx,\;\;\;\;j,k\in\mathbb{Z}.
\end{equation}
Given any finite sequence of nonzero integers $k_{j}$, $j=1,2,3,....,n,$ let $M_n$ be the $n\times n$ matrix,
$$M_{n}:=
   \begin{pmatrix}
     m_{k_{1},k_{1}} & \cdots & m_{k_{1},k_{n}}  \\
     m_{k_{2},k_{1}} & \cdots & m_{k_{2},k_{n}} \\
     \vdots & \vdots & \vdots \\
     m_{k_{n},k_{1}} & \cdots & m_{k_{n},k_{n}} \\
   \end{pmatrix}.
 $$
Then
\begin{itemize}
   \item [(i)]  there exists a constant $\beta>0,$ depending only on $g$, such that
   $$m_{k,k}\geq \beta,\;\;\;\text{ for any}\; k\in\mathbb{Z}-\{0\}.$$
   \item[(ii)] $m_{j,0}=0$, $j\in\mathbb{Z}$.

   \item [(iii)]$M_{n}$ is invertible and hermitian.

   \item [(iv)] there exists $\delta>0,$ depending only on $g$, such that
\begin{equation}\label{posit1}
    \delta_{k}=\|G(\psi_{k})\|^{2}_{L^{2}(\mathbb{T})}> \delta >0,\;\;\text{for all}\;k\in \mathbb{Z}-\{0\}.
\end{equation}
\item [(v)] $m_{-k,k}=\overline{m_{k,-k}}$ and $m_{-k,-k}=\overline{m_{k,k}}.$
 \end{itemize}

\end{lem}

\begin{proof}
The proof of parts (i) and (iii) can be found in  \cite[page 296]{Micu Ortega Rosier and Zhang}.
Part (iv) was proved in \cite[page 3650]{10} (see also  \cite[page 213]{1}). Parts (ii) and (v) are direct consequences of the definition in \eqref{invertmatrix1}.
\end{proof}

Now we give the proof of our first criterion regarding controllability of non-homogenous linear system \eqref{introduc2} stated in Theorem \ref{ControlLa}.

\begin{proof}[Proof of Theorem \ref{ControlLa}] As we already discussed, it suffices to assume $u_{0}=0.$ Let us start by performing a suitable decomposition of $\mathbb{Z}$.  Indeed, in view of $(H3)$ there are only
 finitely many integers in $\mathbb{I},$ say, $k_{j},$ $j=1,2,\cdots,n_{0}^{\ast},$ for some $n_{0}^{\ast}\in \mathbb{N},$
 such that one can find another integer $k\neq k_{j}$
 with $\lambda_{k}=\lambda_{k_{j}}.$ By setting
 $$\mathbb{I}_{j}:=\{k\in \mathbb{Z}: k\neq k_{j}, \lambda_{k}=\lambda_{k_{j}}\},\;\;\;\;j=1,2,\cdots,n_{0}^{\ast},
 $$
we then get the  pairwise disjoint union,
 \begin{equation}\label{zdecomp}
 \mathbb{Z}=\mathbb{I}\cup\mathbb{I}_{1}\cup\mathbb{I}_{2}\cup\cdots\cup\mathbb{I}_{n_{0}^{\ast}}.
 \end{equation}

We now prove the theorem in six steps.

\noindent
	{\bf{Step 1.}} The family $\{e^{-i\lambda_{k}t}\}_{k\in \mathbb{I}}$, with $\lambda_{k}=k a(k)$, is a Riesz basis for $H:=\overline{\text{span}\{e^{-i\lambda_{k}t}: k\in \mathbb{I}\}}$ in $L^{2}([0,T]).$

In fact, since  $L^{2}([0,T])$ is a reflexive separable  Hilbert space so is $H$. In addition, by definition, it is clear that  $\{e^{-i\lambda_{k}t}\}_{k\in \mathbb{I}}$ is complete in $H$.
On the other hand,  from \eqref{gammaseg}-\eqref{gammalinha}  and  Theorem \ref{InghamG2},  there exist positive constants $A$ and $B$  such that
\begin{equation}\label{indesg}
 A\sum_{n \in \mathbb{I}}|b_{n}|^{2}\leq \int_{0}^{T}|f(t)|^{2}dt\leq B\sum_{n \in \mathbb{I}}|b_{n}|^{2},
\end{equation}
for all functions of the form  $f(t)=\sum\limits_{ n \in \mathbb{I}}b_{n}e^{-i\lambda_{n}t},$ $t\in [0,T]$,    with  square-summable complex coefficients $b_{n}.$ In particular, if $b_{1},...,b_{N}$ are $N$ arbitrary  constants we have
 $$\displaystyle{A\sum_{n=1}^{N}|b_{n}|^{2}\leq \left\|\sum_{n=1}^Nb_{n}e^{-i\lambda_{n}t}\right\|^{2}_{H}
 \leq B\sum_{n=1}^{N}|b_{n}|^{2}.}$$
Hence, an application of Theorem \ref{BasRieszTheo} gives the desired property.

\noindent
	{\bf{Step 2.}} There exists a
unique biorthogonal basis $\{q_{j}\}_{j\in \mathbb{I}}\subseteq H^{\ast}$ to $\{e^{-i\lambda_{k}t}\}_{k\in \mathbb{I}}$.

 Indeed, Step 1 and Theorem \ref{BasRieszTheo}  implies that
$\{e^{-i\lambda_{k}t}\}_{k\in \mathbb{I}}$ is a complete Bessel sequence and possesses a biorthogonal system $\{q_{j}\}_{j\in \mathbb{I}}$ which is also
a complete Bessel sequence. Moreover,  Corollary 5.22 in {\cite[page 171]{9}} implies that $\{q_{j}\}_{j\in \mathbb{I}}$ is also a basis for $H$ (after identifying $H^*$ and $H$).
So, from Lemma 5.4 {\cite[page 155]{9}}, we get that $\{e^{-i\lambda_{k}t}\}_{k\in \mathbb{I}}$ is a minimal sequence in $H$; and, hence, exact (see \cite[Definition 5.3]{9}). Finally, Lemma 5.4 in {\cite[page 155]{9}} gives that $\{q_{j}\}_{j\in \mathbb{I}}$ is the unique biorthogonal basis to $\{e^{-i\lambda_{k}t}\}_{k\in \mathbb{I}}$.
Note that an immediate consequence is that
\begin{equation}\label{dualbg}
    (e^{-i\lambda_{k}t}\;,\;q_{j})_{H}=\int_{0}^{T}e^{-i\lambda_{k}t}\overline{q_{j}}(t)\;dt=\delta_{kj},\;\;\;k,j\in \mathbb{I},
\end{equation}
where $\delta_{kj}$ represents the Kronecker delta.

\noindent
	{\bf{Step 3.}}  Here we will define the appropriate control function $h.$

In fact, let $\{q_{j}\}_{j\in \mathbb{I}}$ be the sequence obtained in Step 2. The next step is to extend the sequence  $q_{j}$ for $j$ running on  $\mathbb{Z}$. In view of \eqref{zdecomp} it remains to define this sequence for indices in $\mathbb{I}_j$, $j=1,\cdots, n_{0}^{\ast}$. Furthermore,   $(H2)$ gives that $\mathbb{I}_{j}$ contains at most
$n_{0}-1$ elements. Without loss of generality, we may assume that all multiple eigenvalues have multiplicity $n_0$; otherwise we may repeat the procedure below according to the multiplicity of each eigenvalue. Thus we write
\begin{equation}\label{Ijdef}
\mathbb{I}_{j}=\{k_{j,1}, k_{j,2}, k_{j,3}, \cdots,k_{j,n_{0}-1}\},\;\;\;\;j=1,2,\cdots,n_{0}^{\ast}.
\end{equation}
To simplify notation, here and in what follows we use $k_{j,0}$ for $k_j$. Given $k_{j,l}\in \mathbb{I}_{j}$ we define
$q_{k_{j,l}}:=q_{k_{j,0}}=q_{k_{j}}$. At this point
 recall that $\lambda_{k_{j,l}}=\lambda_{k_{j}}$
 for any $j=1,2,\cdots,n_{0}^{\ast}$ and $l=0,1,2, \cdots, n_{0}-1.$
 
 Having defined $q_j$ for all $j\in\mathbb{Z}$, we now define the control function $h$ by
\begin{equation}\label{thecontrol}
    h (x,t)=\sum_{j \in \mathbb{Z}} h_{j}\;\overline{q_{j}}(t)\;\psi_{j}(x),
\end{equation}
for suitable  coefficients $h_{j}$'s to be determined later.
From the definition of $G$, we obtain
\begin{equation}\label{thecontrol1g}
{\footnotesize
\begin{split}
\int\limits_{0}^{T}
\left(G(h)(x,t),\; e^{-i\lambda_{k}(T-t)} \psi_{k}(x)\right)_{L^{2}_{p}(\mathbb{T})}dt&=
\int\limits_{0}^{T}\left(\sum_{j\in \mathbb{Z}}h_{j}\overline{q_{j}}(t)G(\psi_{j})(x,t),\;
e^{-i\lambda_{k}(T-t)} \psi_{k}(x) \right)_{L^{2}_{p}(\mathbb{T})}dt\\
&=
\sum_{j\in \mathbb{Z}}h_{j}\int\limits_{0}^{T}\overline{q_{j}}(t) e^{i\lambda_{k}(T-t)} \;dt\left(G(\psi_{j})(x),\; \psi_{k}(x)\right)_{L^{2}_{p}(\mathbb{T})}\\
&=\sum_{j\in \mathbb{Z}}h_{j} e^{i\lambda_{k}T} m_{j,k}
\int\limits_{0}^{T}\overline{q_{j}}(t)e^{-i\lambda_{k}t}\,dt,
\end{split}}
\end{equation}
with $m_{j,k}$ defined in \eqref{invertmatrix1}.

\noindent
	{\bf{Step 4.}}
Here we find  $h_{j}$'s such that $h$ defined by \eqref{thecontrol} serves as the required control function. 

First of all, note that in order to prove the first part of the theorem,  identity \eqref{thecontrol1g} and Lemma \ref{coef} yield that it suffices to choose $h_{j}$'s such that
\begin{equation}\label{h formg}
   c_{k}=\sum_{j\in \mathbb{Z}}h_{j}e^{i\lambda_{k}T}m_{j,k}
\int_{0}^{T}\overline{q_{j}}(t)e^{-i\lambda_{k}t}\;dt,
\end{equation}
where $u_{1}(x)=\sum_{n\in \mathbb{Z}}
c_{n}\psi_{n}(x)$.

We will show now that we may indeed choose $h_j$'s satisfying \eqref{h formg}. To see this, first observe that, since $c_0=0$, part (ii)  in Lemma \ref{invertmatrix} implies that \eqref{h formg} holds for $k=0$ independently of $h_{j}$'s. In particular, we may choose $h_0=0$. Next, from \eqref{dualbg}, if $$k\in \widetilde{\mathbb{I}}:=\mathbb{I}-\{k_1,\ldots,k_{n_0^*}\}$$  we see that \eqref{h formg} reduces to
$$c_{k}=h_{k}m_{k,k}\;e^{i\lambda_{k}T}.$$
Hence, in view of part (iii) in Lemma \ref{invertmatrix}, we have
\begin{equation}\label{hform3g}
	h_{k}=\frac{c_{k}\;\; e^{-i\lambda_{k}T}}{m_{k,k}}, \qquad k\in\widetilde{\mathbb{I}}.
\end{equation}
On the other hand, if $k\in\mathbb{Z}-\widetilde{\mathbb{I}}$  then $k=k_{j,l_0}$ for some $j\in\{1, \ldots,n_0^*\}$ and $l_0\in\{0,1,\ldots,n_0-1\}$. Since $\lambda_{k}=\lambda_{k_{j,l_0}}=\lambda_{k_j}$, the integral in \eqref{h formg} is zero, except for those indices in  $\mathbb{I}_{j}\cup\{k_j\}$. In particular, \eqref{h formg} reduces to
\begin{equation}\label{ckl}
	\displaystyle{c_{k_{j,l_0}}=c_k=\sum_{l=0}^{n_{0}-1}h_{k_{j,l}} m_{k_{j,l},k_{j,l_0}}e^{i\lambda_{k_{j,l_0}}T}}.
\end{equation}
When $l_0$ runs over the set $\{0,1,\ldots,n_0-1\}$, the equations in \eqref{ckl} may be seen as a linear system for $h_{k_{j,l}}$ (with $j$ fixed) whose unique solution is
\begin{equation}\label{hform4g}
	{
		\begin{pmatrix}
			h_{k_{j,0}}\\
			h_{k_{j,1}}\\
			\vdots\\
			h_{k_{j,n_{0}-1}}
		\end{pmatrix}^{\top}=
		\begin{pmatrix}
			c_{k_{j,0}} e^{-i\lambda_{k_{j,0}}T}\\
			c_{k_{j,1}} e^{-i\lambda_{k_{j,1}}T}\\
			\vdots\\
			c_{k_{j,n_{0}-1}} e^{-i\lambda_{k_{j,n_{0}-1}}T}\\
		\end{pmatrix}^{\top}
		M_{j}^{-1},\;\;\text{for} \; j=1,2,\cdots,n_{0}^{\ast},}
\end{equation}

where
$$M_{j}=
  \begin{pmatrix}
    m_{k_{j,0},k_{j,0}} & m_{k_{j,0},k_{j,1}} &  \cdots& m_{k_{j,0},k_{j,n_{0}-1}} \\
    m_{k_{j,1},k_{j,0}} & m_{k_{j,1},k_{j,1}}  &   \cdots & m_{k_{j,1},k_{j,n_{0}-1}} \\
    \vdots&  &  \ddots &\vdots\\
m_{k_{j,n_{0}-1},k_{j,0}} & m_{k_{j,n_{0}-1},k_{j,1}}  &   \cdots & m_{k_{j,n_{0}-1},k_{j,n_{0}-1}} \\
  \end{pmatrix}
.$$
Since from Lemma \ref{invertmatrix} the matrix $M_j$ is invertible, equation \eqref{hform4g} makes sense. Consequently, for any $j\in\mathbb{Z}=\mathbb{I}\cup\mathbb{I}_{1}\cup\mathbb{I}_{2}\cup\cdots\cup\mathbb{I}_{n_{0}^{\ast}}$, we may choose $h_{j}$'s according to \eqref{hform3g} and
\eqref{hform4g}.

\noindent
	{\bf{Step 5.}}
The function $h$ defined by
\eqref{thecontrol} with $h_{0}=0$ and $h_{k}$, $k\neq 0$, given  by \eqref{hform3g} and \eqref{hform4g} 
belongs to $L^{2}([0,T];H_{p}^{s}(\mathbb{T}))$.

 Indeed, recall from Step 2 that $\{q_{j}\}_{j\in \mathbb{I}}$ is a Riesz basis for $H$. Thus, from Theorem \ref{BasRieszTheo} part (3), it follows that $\{q_{j}\}_{j\in \mathbb{I}}$ is a bounded sequence in $L^2([0,T])$. Consequently, $\{q_{j}\}_{j\in \mathbb{Z}}$ is also bounded in $L^2([0,T])$. Hence, by using the explicit representation in \eqref{thecontrol}, we deduce
\begin{align}\label{hes}
\begin{split}
\|h\|^{2}_{L^{2}([0,T];H_{p}^{s}(\mathbb{T}))}
&=\frac{1}{2\pi}\sum_{k\in \mathbb{Z}} (1+|k|)^{2s}|h_{k}|^2 \int_{0}^{T}|q_{k}(t)|^{2}\;dt\\
&\leq C \sum_{k\in \mathbb{Z}}(1+|k|)^{2s}|h_{k}|^{2},
\end{split}
\end{align}
for some positive constant $C$.
Thus, from identity \eqref{hform3g} and Lemma \ref{invertmatrix} part (ii)  we obtain
\begin{equation}\label{hform5g}
\begin{split}
\|h\|^{2}_{L^{2}([0,T];H_{p}^{s}(\mathbb{T}))}
&\leq C
{\sum_{ k\in \widetilde{\mathbb{I}},k\neq 0}} (1+|k|)^{2s}\left|\frac{c_{k} e^{-i\lambda_{k}T}}{m_{k,k}}\right|^{2}
+ C \sum_{k\in\mathbb{Z}-\widetilde{\mathbb{I}}}(1+|k|)^{2s}|h_{k}|^{2} \\
&\leq \frac{C}{\beta^{2}}{\sum_{ k\in \widetilde{\mathbb{I}},k\neq 0}} (1+|k|)^{2s}\left|c_{k}\right|^{2}
+ C \sum_{k\in\mathbb{Z}-\widetilde{\mathbb{I}}}(1+|k|)^{2s}|h_{k}|^{2}.
\end{split}
\end{equation}
Since $u_1\in H^s_p(\mathbb{T})$ the above series converges. In addition, since the set  $\mathbb{Z}-\widetilde{\mathbb{I}}$ is finite we conclude that the right-hand side of \eqref{hform5g} is finite, implying that $h$ belongs to $L^{2}([0,T];H_{p}^{s}(\mathbb{T}))$.

In order to complete the proof of the theorem it remains to establish \eqref{hboun}.

\noindent
{\bf{Step 6.}} Estimate \eqref{hboun} holds.

From Step 5 we see that we need to estimate de second term on the right-hand side of  \eqref{hform5g}. So, fix some nonzero $k\in\mathbb{Z}-\widetilde{\mathbb{I}}$. We may write $k=k_{j,l}$ for some $l=0,1,2,\cdots,n_{0}-1$ and $j=1,2,\cdots,n_{0}^{\ast}$. From \eqref{hform4g} we infer
\begin{align*}
\begin{split}
|h_{k_{j,l}}|^{2}&\leq \sum_{m=0}^{n_{0}-1}|h_{k_{j,m}}|^{2}
\leq \left( \sum_{m=0}^{n_{0}-1}\left|c_{k_{j,m}}e^{-i\lambda_{k_{j,m}}T}\right|^{2} \right)\|M_{j}^{-1}\|^{2}
\leq \|M_{j}^{-1}\|^{2}\sum_{m=0}^{n_{0}-1}|c_{k_{j,m}}|^{2},
\end{split}
\end{align*}
where $\|M_{j}^{-1}\|$ is the Euclidean norm of the matrix $M_{j}^{-1}.$
This implies that
\begin{equation*}
\begin{split}
(1+|k_{j,l}|)^{2s}|h_{j,l}|^{2}
& \leq\sum_{m=0}^{n_{0}-1}\|M_{j}^{-1}\|^{2}
\frac{(1+|k_{j,l}|)^{2s}}{(1+|k_{j,m}|)^{2s}}(1+|k_{j,m}|)^{2s}|c_{k_{j,m}}|^{2}\\
 &\leq C(s)\sum_{m=0}^{n_{0}-1}(1+|k_{j,m}|)^{2s}|c_{k_{j,m}}|^{2},
\end{split}
\end{equation*}
with $$\displaystyle{C(s)=\underset{m,l=0,1,2,\cdots, n_{0}-1}{\max_{j=1,2,...,n_{0}^{\ast}}}
 \left\{\|M_{j}^{-1}\|^{2}\frac{(1+|k_{j,l}|)^{2s}}{(1+|k_{j,m}|)^{2s}}\right\} }.$$
Therefore,
\begin{equation}\label{hform7g}
\begin{split}
\sum_{k\in\mathbb{Z}-\widetilde{\mathbb{I}}}(1+|k|)^{2s}|h_{k}|^{2}&=\sum_{j=1}^{n_0^*}\sum_{l=0}^{n_0-1}(1+|k_{j,l}|)^{2s}|h_{k_{j,l}}|^2\\
&\leq C(s)n_0\sum_{m=1}^{n_0^*}\sum_{l=0}^{n_0-1}(1+|k_{j,m}|)^{2s}|c_{k_{j,m}}|^2\\
&=C(s)n_0\sum_{k\in\mathbb{Z}-\widetilde{\mathbb{I}}}(1+|k|)^{2s}|c_{k}|^{2}.
\end{split}
\end{equation}
Gathering together \eqref{hform5g} and \eqref{hform7g}, we obtain
\begin{equation}\label{cota}
\begin{split}
\|h\|^{2}_{L^{2}([0,T];H_{p}^{s}(\mathbb{T}))}
&\leq \frac{C}{\beta^{2}}{\sum_{ k\in \widetilde{\mathbb{I}},k\neq 0}} (1+|k|)^{2s}\left|c_{k}\right|^{2}
+ C C(s)n_0\sum_{k\in\mathbb{Z}-\widetilde{\mathbb{I}}}(1+|k|)^{2s}|c_{k}|^{2}\\
&\leq \nu^2\|u_1\|^2_{H^s_p(\mathbb{T})},
\end{split}
\end{equation}
where $\displaystyle{\nu^{2}=\max\left\{\frac{ C }{\beta^{2}},\; n_{0}CC(s)\right\}}$.

 This completes the proof of the theorem.
\end{proof}

Now we  proof  our second criterion regarding controllability of non-homogenous linear system \eqref{introduc2}  stated in Theorem \ref{ControlLag}.

\begin{proof}[Proof of Theorem \ref{ControlLag}]
The proof  is similar to that of Theorem \ref{ControlLa}. So we bring only the necessary changes and estimates. As before, we assume $u_{0}=0.$ In view of $(H4),\;(H5)$ we may find
 finitely many integers in $\mathbb{J},$ say, $k_{j},$ $j=1,2,\cdots,n_{0}^{\ast},$ for some $n_{0}^{\ast}\in \mathbb{N},$ with $n_{0}^{\ast}\leq 2 k_{1}^{\ast}-1,$
 such that one can find another integer $k\neq k_{j}$
 with $\lambda_{k}=\lambda_{k_{j}}.$ Let
 $$\mathbb{J}_{j}:=\{k\in \mathbb{Z}: k\neq k_{j}, \lambda_{k}=\lambda_{k_{j}}\},\;\;\;\;j=1,2,\cdots,n_{0}^{\ast},$$
and
$$\mathbb{J}^{-}:= \{ k\in \mathbb{Z}: k\leq -k_{1}^{\ast} \}.$$
Then we obtain the pairwise disjoint decomposition
 \begin{equation}\label{zde1}
 \mathbb{Z}=\mathbb{J}^{-} \cup\mathbb{J}\cup\mathbb{J}_{1}\cup\mathbb{J}_{2}\cup\cdots\cup\mathbb{J}_{n_{0}^{\ast}}.
 \end{equation}

Again, we may prove the theorem into six steps.

\noindent
	{\bf{Step 1.}} The family $\{e^{-i\lambda_{k}t}\}_{k\in \mathbb{J}}$, with $\lambda_{k}=k a(k)$, is a Riesz basis for $H:=\overline{\text{span}\{e^{-i\lambda_{k}t}: k\in \mathbb{J}\}}$ in $L^{2}([0,T])$.

In fact, this is a consequence of \eqref{gammasegg}-\eqref{gammalinhag},  Theorem \ref{InghamG2}, and Theorem \ref{BasRieszTheo}.

\noindent
	{\bf{Step 2.}} There exists a
	unique biorthogonal basis $\{q_{j}\}_{j\in \mathbb{J}}\subseteq H^{\ast}$ to $\{e^{-i\lambda_{k}t}\}_{k\in \mathbb{J}}$.

 This is a consequence of Theorem \ref{BasRieszTheo},  Corollary 5.22 in \cite{9}, and  Lemma 5.4 \cite{9}. Furthermore, we have that
\begin{equation}\label{dualbgg}
    (e^{-i\lambda_{k}t}\;,\;q_{j})_{H}=\int_{0}^{T}e^{-i\lambda_{k}t}\overline{q_{j}}(t)\;dt=\delta_{kj},\;\;\;k,j\in \mathbb{J}.
\end{equation}

\noindent
	{\bf{Step 3.}}  Here we will define an adequate control function $h.$
	
	As in \eqref{thecontrol}, for suitable coefficients $h_j$ to be determined later we set
	\begin{equation}\label{thecontrolg}
	h(x,t)=\sum_{j \in \mathbb{Z}} h_{j}\;\overline{q_{j}}(t)\;\psi_{j}(x),
	\end{equation}
where, according to the decomposition \eqref{zde1}, the sequence $\{q_k\}_{k\in\mathbb{Z}}$ is defined as follows: if $k\in \mathbb{J}$ then $q_k$ is given in Step 2; if $k\in \mathbb{J}_j$ for some $j\in\{1,\ldots,n_0^*\}$ then by writing (assuming that all multiple eigenvalues have multiplicity $n_0$)
$$
\mathbb{J}_{j}=\{k_{j,1}, k_{j,2}, k_{j,3}, \cdots,k_{j,n_{0}-1}\},
$$
and denoting $k_{j}$ by $k_{j,0}$ we set
$$
q_k=q_{k_{j,l}}:=q_{k_{j,0}}=q_{k_{j}}.
$$
Finally, if $k\in \mathbb{J}^-$ then we set
$$
q_k=q_{-k}.
$$
With this choice of $\{q_k\}_{k\in\mathbb{Z}}$, as in \eqref{thecontrol1g} we have
\begin{equation}\label{thecontrol1gg}
\int_{0}^{T}\left(G(h)(x,t),\; e^{-i\lambda_{k}(T-t)} \psi_{k}(x)\right)_{L^{2}_{p}(\mathbb{T})}dt
=\sum_{j\in \mathbb{Z}}h_{j} e^{i\lambda_{k}T} m_{j,k}
\int_{0}^{T}\overline{q_{j}}(t)e^{-i\lambda_{k}t}\;dt.
\end{equation}

\noindent
	{\bf{Step 4.}}
 In this step we find  $h_{j}$'s such that $h$ defined by \eqref{thecontrolg} serves as the required control function.

 By writing 
$u_{1}(x)=\sum_{n\in \mathbb{Z}}
c_{n}\psi_{n}(x)$, it is enough to consider $h_{j}$'s satisfying
\begin{equation}\label{h formgg}
   c_{k}=\sum_{j\in \mathbb{Z}}h_{j}e^{i\lambda_{k}T}m_{j,k}
\int_{0}^{T}\overline{q_{j}}(t)e^{-i\lambda_{k}t}\;dt.
\end{equation}
From Lemma \ref{invertmatrix} part (ii) we may take $h_0=0$. To see that we can choose $h_j$ such that \eqref{h formgg} holds let us start by defining the following sets of indices
$$
\mathbb{J}^+:=\left\{k\in \mathbb{Z}\,:\, k\geq k_1^*\right\},
$$
$$
\widetilde{\mathbb{J}}=\left\{k\in \mathbb{Z}\,: k= k_{j,l};\;\;
l=0,1,2,\cdots,n_{0}-1,\;\;\;j=1,2,\cdots,n_{0}^{\ast}\right\},
$$
and
$$
\widetilde{\mathbb{I}}=\left\{k\in \mathbb{Z}\,: \, k\notin  \mathbb{J}^{+}\cup \mathbb{J}^{-} \;\text{and}\;k\notin\widetilde{\mathbb{J}} \right\}.
$$
It is clear that $\mathbb{Z}=\widetilde{\mathbb{I}}\cup \widetilde{\mathbb{J}}\cup \mathbb{J}^+\cup \mathbb{J}^-$. In addition, note that $\widetilde{\mathbb{I}}$ is nothing but the set of those indices for which the corresponding eigenvalue is simple. Without loss of generality we will assume that $\widetilde{\mathbb{I}}$ is nonempty; otherwise, this part has no contribution and these indices do not appear in \eqref{h formgg}.

The idea now is to obtain $h_k$ according to $k\in \widetilde{\mathbb{I}}$, $k\in \widetilde{\mathbb{J}}$, or $k\in \mathbb{J}^+\cup \mathbb{J}^-$.
  From \eqref{dualbgg} we see that \eqref{h formgg} reduces to
$$c_{k}=h_{k}m_{k,k}\;e^{i\lambda_{k}T},\;\;\;k\in \widetilde{\mathbb{I}}.$$
Therefore, 
\begin{equation}\label{hform3gg}
h_{k}=\frac{c_{k}\;\; e^{-i\lambda_{k}T}}{m_{k,k}},\;\;k\in \widetilde{\mathbb{I}}.
\end{equation}

Next, if $k\in \widetilde{\mathbb{J}}$,  then $k=k_{j,l_0}$ for some $j\in\{1, \ldots,n_0^*\}$ and $l_0\in\{0,1,\ldots,n_0-1\}$. Thus, as in Step 4 of Theorem \ref{ControlLa}, we see that
$${c_{k_{j,l_0}}=\sum_{l=0}^{n_{0}-1}h_{k_{j,l}} m_{k_{j,l},k_{j,l_0}}e^{i\lambda_{k_{j,l_0}}T}}, $$
By solving the above system for $h_{k_{j,l}}$ (with $j$ fixed and running $l_0$ over $\{0,1,\ldots,n_0-1\}$) we find
\begin{equation}\label{hform4gg}
	\begin{pmatrix}
	h_{k_{j,0}}\\
	h_{k_{j,1}}\\
	\vdots\\
	h_{k_{j,n_{0}-1}}
	\end{pmatrix}^{\top}=
	\begin{pmatrix}
	c_{k_{j,0}} e^{-i\lambda_{k_{j,0}}T}\\
	c_{k_{j,1}} e^{-i\lambda_{k_{j,1}}T}\\
	\vdots\\
	c_{k_{j,n_{0}-1}} e^{-i\lambda_{k_{j,n_{0}-1}}T}\\
	\end{pmatrix}^{\top}
	\tilde{M}_{j}^{-1},\;\; \; j=1,2,\cdots,n_{0}^{\ast},
\end{equation}
where
$$\tilde{M}_{j}=
\begin{pmatrix}
m_{k_{j,0},k_{j,0}} & m_{k_{j,0},k_{j,1}} &  \cdots& m_{k_{j,0},k_{j,n_{0}-1}} \\
m_{k_{j,1},k_{j,0}} & m_{k_{j,1},k_{j,1}} &   \cdots & m_{k_{j,1},k_{j,n_{0}-1}} \\
\vdots&  &  \ddots &\vdots\\
m_{k_{j,n_{0}-1},k_{j,0}} & m_{k_{j,n_{0}-1},k_{j,1}}  &   \cdots & m_{k_{j,n_{0}-1},k_{j,n_{0}-1}} \\
\end{pmatrix}.$$
Finally, if $k\in \mathbb{J}^{+}$ we have $-k \in \mathbb{J}^{-}$ and $I(k)=\{k,-k\}$. We deduce from \eqref{h formgg} that
\[
\begin{cases}
\displaystyle{c_{k}=h_{k} e^{i\lambda_{k}T}  m_{k,k} + h_{-k} e^{i\lambda_{k}T}  m_{-k,k}},    \\
\displaystyle{c_{-k}=h_{k} e^{i\lambda_{-k}T}  m_{k,-k} + h_{-k} e^{i\lambda_{-k}T}  m_{-k,-k}}.
\end{cases}
\]
Solving this system for $h_k$ and $h_{-k}$ we obtain
\begin{equation}\label{even}
\begin{pmatrix}
 h_{k}\\
h_{-k}\\
\end{pmatrix}^{\top}
=\begin{pmatrix}
c_{k} e^{-i\lambda_{k}T}\\
c_{-k} e^{-i\lambda_{-k}T}\\
\end{pmatrix}^{\top}M^{-1},\quad  k\in \mathbb{J}^{+}
\end{equation}
where
$$M^{-1}=\frac{1}{d_k}
  \begin{pmatrix}
{ m_{-k,-k} }    &  -{ m_{k,-k} }  \\
-{ m_{-k,k} }    &  { m_{k,k} }  \\
  \end{pmatrix}, \qquad d_{k}=m_{k,k}m_{-k,-k}-m_{k,-k}m_{-k.k}. $$
Summarizing the above construction, we see that, for any $k\in \mathbb{Z}$, we may choose $h_{k}$ according to \eqref{hform3gg},
\eqref{hform4gg}, and \eqref{even}.

Next we observe that the matrix $M^{-1}$ is bounded uniformly with respect to $k$. Indeed,
 from Lemma \ref{invertmatrix} part (v) we infer that $d_{k}=|m_{k,k}|^{2}-|m_{k,-k}|^{2}.$ Now, from the definition of $G$,
$$
{\normalsize
\begin{split}
m_{k,-k}&= \frac{1}{2\pi}\left[\int_{0}^{2\pi}g(x)e^{i2kx}dx- \left(\int_{0}^{2\pi}g(x)e^{ikx}dx\right) \left(\int_{0}^{2\pi}g(y)e^{iky}dy\right)\right]\\
&=\widehat{g}(-2k)-2\pi[\widehat{g}(-k)]^2.
\end{split}}
$$
Since $g$ is smooth, using the Riemann-Lebesgue lemma we obtain
$$
\lim_{k\rightarrow +\infty}|m_{k,-k}|=0.
$$
On the other hand, in view of \eqref{gcondition},   for any $k\neq0$,
$$m_{k,k}=\frac{1}{2\pi}\left(\int_{0}^{2\pi}g(x)dx-\left|\int_{0}^{2\pi}g(x)e^{ikx}dx\right|^{2}\right)=\frac{1}{2\pi}-|\widehat{g}(-k)|^2,$$ 
and, hence,
$$
\lim_{k\rightarrow +\infty}m_{k,k}=\frac{1}{2\pi}.
$$
Since
$$
\lim_{k\rightarrow +\infty}d_k=\frac{1}{4\pi^2},
$$
we may assume, without loss of generality, that $d_k>\frac{1}{8\pi^2}$, for any $k\geq k_1^*$. Therefore, there exists $D>0$, independent of $k\in \mathbb{J}^+$, such that
\begin{equation}\label{determ2}
    \|M^{-1}\|\leq D,
\end{equation}
where $\|M^{-1}\|$ is the Euclidean norm of the matrix $M^{-1}.$

\noindent
	{\bf{Step 5.}}
The control function $h$ defined by \eqref{thecontrolg} with $h_{0}=0,$ and $h_{k}$, $k\neq0$, given by \eqref{hform3gg}, \eqref{hform4gg}, and \eqref{even} belongs to $L^{2}([0,T];H_{p}^{s}(\mathbb{T}))$.

 Indeed, as in \eqref{hes} we obtain
\begin{align*}
\begin{split}
\|h\|^{2}_{L^{2}([0,T];H_{p}^{s}(\mathbb{T}))}
&\leq C \sum_{k\in \mathbb{Z}}(1+|k|)^{2s}\;|h_{k}|^{2},
\end{split}
\end{align*}
for some positive constant $C$. Next, in the series above we split the sum according to $k\in \widetilde{\mathbb{I}}$,  $k\in \widetilde{\mathbb{J}}$ or $k\in  \mathbb{J}^{+}\cup \mathbb{J}^{-}$.
Thus, we may write
\begin{equation}\label{hform5gg}
\begin{split}
\|h\|^{2}_{L^{2}([0,T];H_{p}^{s}(\mathbb{T}))}
&\leq C
{\sum_{ k\in \widetilde{\mathbb{I}}}} (1+|k|)^{2s}\left|h_{k}\right|^{2}
+{\sum_{ k\in \widetilde{\mathbb{J}}}} (1+|k|)^{2s}\left|h_{k}\right|^{2}\\
&\quad + C \sum_{k\in  \mathbb{J}^{+}\cup\mathbb{J}^{-} }(1+|k|)^{2s}|h_{k}|^{2}.
\end{split}
\end{equation}
The first two terms on the right-hand side of \eqref{hform5gg} may be estimated as in Theorem \ref{ControlLa} (see \eqref{hform7g}). Thus,
\begin{equation}\label{hform5ggg}
	\begin{split}
		\|h\|^{2}_{L^{2}([0,T];H_{p}^{s}(\mathbb{T}))}
		&\leq \frac{C}{\beta^2}
		{\sum_{ k\in \widetilde{\mathbb{I}}}} (1+|k|)^{2s}\left|c_{k}\right|^{2}
		+ CC(s)n_0 	{\sum_{ k\in \widetilde{\mathbb{J}}}} (1+|k|)^{2s}\left|c_{k}\right|^{2}\\
		&\quad + C \sum_{k\in  \mathbb{J}^{+}\cup\mathbb{J}^{-} }(1+|k|)^{2s}|h_{k}|^{2},
	\end{split}
\end{equation}
where $\displaystyle{C(s)=\underset{m,l=0,1,2,\cdots, n_{0}-1}{\max_{j=1,2,...,n_{0}^{\ast}}}
 \left\{\|\tilde{M}_{j}^{-1}\|^{2}\frac{(1+|k_{j,l}|)^{2s}}{(1+|k_{j,m}|)^{2s}}\right\} }.$\\

For the last term in \eqref{hform5ggg}, identity \eqref{even} and \eqref{determ2} imply that for any $ k\in  \mathbb{J}^{+}$,
\begin{align*}
\begin{split}
|h_{k}|^{2}&
\leq \left( \left|c_{k}\right|^{2} + \left|c_{-k}\right|^{2}\right)\|M^{-1}\|^{2} \leq
\left( \left|c_{k}\right|^{2} + \left|c_{-k}\right|^{2}\right)D^{2},
\end{split}
\end{align*}
and
\begin{equation}\label{even1}
\begin{split}
(1+|k|)^{2s}|h_{k}|^{2}&\leq
D^{2} (1+|k|)^{2s}\left|c_{k}\right|^{2} + D^{2} 
(1+|-k|)^{2s}  \left|c_{-k}\right|^{2}.
\end{split}
\end{equation}
Since the right-hand side of \eqref{even1} is symmetric with respect to $k$, the same estimate holds for $k\in \mathbb{J}^{-},$ from which we deduce that
\begin{equation}\label{even2}
	\sum_{k\in  \mathbb{J}^{+}\cup\mathbb{J}^{-} }(1+|k|)^{2s}|h_{k}|^{2}\leq 2D^2 \sum_{k\in  \mathbb{J}^{+}\cup\mathbb{J}^{-} }(1+|k|)^{2s}|c_{k}|^{2}.
\end{equation}
Combining \eqref{hform5ggg} with \eqref{even1} we get that $h$ belongs to $L^{2}([0,T];H_{p}^{s}(\mathbb{T}))$.

\noindent
{\bf{Step 5.}} Estimate \eqref{hboun} holds.

In view of \eqref{hform5ggg} and \eqref{even2}, we obtain
\begin{equation*}
	\begin{split}
		\|h\|^{2}_{L^{2}([0,T];H_{p}^{s}(\mathbb{T}))}
		&\leq \frac{C}{\beta^2}
		{\sum_{ k\in \widetilde{\mathbb{I}}}} (1+|k|)^{2s}\left|c_{k}\right|^{2}
		+ CC(s)n_0 	{\sum_{ k\in \widetilde{\mathbb{J}}}} (1+|k|)^{2s}\left|c_{k}\right|^{2}\\
		&\quad + 2D^2C \sum_{k\in  \mathbb{J}^{+}\cup\mathbb{J}^{-} }(1+|k|)^{2s}|c_{k}|^{2}\\
		&\leq  \nu^2\,\|u_1\|_{H_{p}^{s}(\mathbb{T})}^2,
	\end{split}
\end{equation*}
where $\nu^{2}=\max\left\{\frac{C}{\beta^{2}}, \;n_{0}CC(s),\; 2CD^{2}\right\}$.

 This completes the proof of the theorem.
\end{proof}

\begin{rem}
The dependence of $\nu$ with respect to $T$ is implicit in the constant $C$ which may depend on the time $T$.
\end{rem}

As an immediate consequence of Theorems \ref{ControlLa} and \ref{ControlLag} we get the following corollary.

\begin{cor} \label{controloperator}
For $s\in \mathbb{R}$ and $T> \frac{2\pi}{\gamma'}$ given, there exists a unique
bounded linear  operator
$$\left\{\begin{array}{lclc}
	\Phi:& H_{p}^{s}(\mathbb{T})\times H_{p}^{s}(\mathbb{T}) & \longrightarrow & L^{2}([0,T];H_{p}^{s}(\mathbb{T}))\\
	&  (u_0,u_1)&  \longmapsto &\Phi(u_{0},u_{1}):=h\\
\end{array}
\right.
$$
such that
\begin{equation}\label{cont}
u_{1}=U(T)u_{0}+\int_{0}^{T}U(T-s)(G(\Phi(u_{0},u_{1})))(\cdot,s)\;ds
\end{equation}
 and
\begin{equation}\label{oprestima}
\|\Phi(u_{0},u_{1})\|_{L^{2}([0,T];H_{p}^{s}(\mathbb{T}))} \leq \nu\; (\|u_{0}\|_{H_{p}^{s}(\mathbb{T})}
+\|u_{1}\|_{H_{p}^{s}(\mathbb{T})}),
\end{equation}
for some positive constant $\nu$.
\end{cor}

We end this section recalling Corollary  \ref{controloperator1} to obtain the  observability inequality, which in turn plays a fundamental role to get the exponential asymptotic stabilization with arbitrary decay rate.

 \begin{cor} \label{controloperator1g}
Let $s\in\mathbb{R}$ and $T> \frac{2\pi}{\gamma'}$ be given. There exists $\delta>0$ such that
$$\int_{0}^{T}\|G^{\ast}U(\tau)^{\ast}\phi^*\|
    ^{2}_{(H^{s}_{p}(\mathbb{T}))'}(\tau) \; d\tau \geq \delta^{2}
    \|\phi^*\|^{2}_{(H^{s}_{p}(\mathbb{T}))'}, $$
    for any $\phi^* \in (H^{s}_{p}(\mathbb{T}))'.$
 \end{cor}
\begin{rem}
    If $\gamma'=+\infty$ or $\tilde{\gamma}'=+\infty$ then Corollaries \ref{controloperator} and \ref{controloperator1g} are valid for any positive time $T.$
\end{rem}

\section{Proof of Theorems \ref{st351} and \ref{estabilization}}\label{section6}
This section is devoted to prove the exponential stabilization results. Once we have the observability inequality in Corollary \ref{controloperator1g} it is well known that this implies the stabilization. So, we just give the main steps. Fist recall we are dealing with the equation
\begin{equation}\label{eq1}
	 \partial_{t}u= \partial_{x}\mathcal{A}u+Gh.
\end{equation}
Since any solution of \eqref{eq1} preserves its mass, without loss of generality, one can assume that the initial data $u_{0}$ satisfies $\widehat{u_{0}}(0)=0$ (otherwise, we perform the change of variables $\tilde{u}=u-\widehat{u_{0}}(0)$). Thus, it is enough to  study the stabilization problem  in $H_{0}^{s}(\mathbb{T})$, $s\in \mathbb{R}.$

The idea to prove Theorems \ref{st351} and \ref{estabilization} is to show the existence of a bounded linear operator, say, $K_1$ on $H_{0}^{s}(\mathbb{T})$ such that
$$
h=K_1u
$$
serves as the feedback control law. So, we study the stabilization problem for the system
\begin{equation}\label{atabilizationL2}
\begin{cases}
u\in C([0,+\infty);H_{0}^{s}(\mathbb{T}))\\
\partial_{t}u=\partial_{x}\mathcal{A}u+GK_1u \in H_{0}^{s-r}(\mathbb{T}),\quad  t>0,\\
u(0)=u_{0}\in H_{0}^{s}(\mathbb{T}), 
\end{cases}
\end{equation}

First, we prove that  system \eqref{atabilizationL2} is globally well-posed in $H_{0}^{s}(\mathbb{T})$, $s\in \mathbb{R}$.

\begin{thm}\label{solsta1}
	Let $u_{0}\in H_{0}^{r}(\mathbb{T}),$ with $r$ as in \eqref{limite}. Then the IVP \eqref{atabilizationL2}
	has a unique solution
	$$u\in C([0,\infty);H_{0}^{r}(\mathbb{T}))
	\cap C^{1}([0,\infty);L^{2}_{0}(\mathbb{T})).$$
Moreover, if $u_{0}\in H_{0}^{s}(\mathbb{T}),$ then we have  $u\in C([0,\infty);H_{0}^{s}(\mathbb{T})),$ for any $s \in \mathbb{R}.$
\end{thm}
\begin{proof}
	Since
	$\partial_{x}\mathcal{A}$
	is the infinitesimal generator of a
	$C_{0}$-semigroup $\{U(t)\}_{t\geq 0}$ in $H_{0}^{s}(\mathbb{T})$ and $GK_1$ is a bounded linear operator on $H^{s}_{0}(\mathbb{T})$, we have that
	$\partial_{x}\mathcal{A}+GK_1$  is also an infinitesimal generator of a $C_{0}-$semigroup  on $H^{s}_{0}(\mathbb{T})$ (see, for instance,  \cite[page 76]{4}). Thus this a consequence of the semigroup theory.
\end{proof}

As we will see, Theorems \ref{st351} and \ref{estabilization} are consequences of the following result.

\begin{thm}\label{st35}
	Let $s\in\mathbb{R}$    be given and $g$ as in \eqref{gcondition}. For any given $\lambda>0$, there exist a bounded linear operator  $K_{1}$ on $H_{0}^{s}(\mathbb{T}))$ such that the unique solution of the closed-loop system
	\begin{equation}\label{atabilizationL2gg}
\begin{cases}
\partial_{t}u=\partial_{x}\mathcal{A}u+GK_1u,  \\
u(0)=u_{0}, 
\end{cases}
\end{equation}
 satisfies
\begin{equation}\label{c5}
	\|u(\cdot,t)\|_{H_{0}^{s}(\mathbb{T})}\leq M
	e^{-\lambda t}\|u_{0}\|_{H_{0}^{s}(\mathbb{T})},\;\;\;\text{for all}\;\;
	t\geq 0.
	\end{equation}
	where the positive constant $M$ depends on $s$ and $G$ but is independent of $u_{0}.$	
\end{thm}
\begin{proof}
    This is a consequence of Corollary \ref{controloperator1g} and the classical principle that exact controllability implies exponential stabilizability for conservative control systems (see Theorem 2.3/Theorem 2.4 in \cite{Liu} and Theorem 2.1 \cite{Slemrod}). To be more precise, according to  \cite{Slemrod, Liu}, one can choose
$$K_{1}=-G^{\ast}L^{-1}_{T,\lambda},$$
where, for some $T>\frac{2\pi}{\gamma'}$,
\begin{equation*}
L_{T,\lambda}\phi=\int_{0}^{T}e^{-2\lambda \tau}\;U(-\tau)GG^{\ast}
U(-\tau)^{\ast}\phi\;d\tau,
\;\;\;\;\;\phi\in H_{0}^{s}(\mathbb{T}),
\end{equation*}
and $U(t)$ is the $C_0$-semigroup generated by $\partial_{x}\mathcal{A}$
(see Lemma 2.4 in \cite{14} for more details). In addition, if one simply chooses $K_{1}=-G^{\ast}$ then there exists $\alpha>0$ such that estimate \eqref{c5} holds with $\lambda$ replaced by $\alpha.$
\end{proof}

Finally, observe that Theorem \ref{st351} and Theorem \ref{estabilization} are  direct consequences of Theorem \ref{st35} just by taking $K_{1}=-G^{\ast}$ and $K_{1}=-G^{\ast}L^{-1}_{T,\lambda}$, respectively.

\section{Applications}\label{section6g}

As an application of our results, we will establish the controllability and stabilization for some linearized dispersive equations of the form \eqref{BO}.

\subsection{The linearized Smith equation}

  The nonlinear Smith equation posed on the entire real line reads as
\begin{equation}\label{ILW}
	\partial_{t}u-\partial_{x}\mathcal{A}u+u\partial_{x}u=0, \;\;\;\;x\in \mathbb{R},\;\;t\in \mathbb{R},
\end{equation}
where $u=u(x,t)$ denotes a real-valued function and  $\mathcal{A}$ is the nonlocal operator  defined by
$$
\widehat{\mathcal{A}u}(\xi):=2\pi\left(\sqrt{\xi^{2}+1}-1\right)\widehat{u}(\xi).$$
Here the hat stands for the Fourier transform on the line. Equation \eqref{ILW} was
derived by Smith in \cite{Smith} and it governs certain types of continental-shelf waves. From the mathematical viewpoint, the well-posedness of the IVP associated to \eqref{ILW}  in $H^{s}(\mathbb{R})$  has been studied for instance in \cite{Abdelouhab Bona Felland and Saut}, \cite{iorio}, and \cite{6}. In \cite[Theorems 7.1 and 7.7]{Abdelouhab Bona Felland and Saut} the authors proved that \eqref{ILW} is globally well-posed in $H^{s}(\mathbb{R})$ for $s=1$ and $s\geq 3/2$. In \cite{iorio} the author established a global well-posedness result in the weighted Sobolev space $H^s(\mathbb{R})\cap L^2((1+|x|^2)^sdx)$ for $s>3/2$.

The control equation associated with the linearized Smith equation on the periodic setting reads as
\begin{equation}\label{ILWP}
	\partial_{t}u-\partial_{x}\mathcal{A}u=Gh, \;\;\;\;x\in \mathbb{T},\;\;t\in \mathbb{R},
\end{equation}
where $\mathcal{A}$ is such that
\begin{equation}\label{ILWP1}
	\widehat{\mathcal{A}u}(k):= 2\pi\left(\sqrt{k^{2}+1}-1\right) \widehat{u}(k),\;\;\;\;k\in \mathbb{Z},
\end{equation}
so that $a(k)=2\pi\left(\sqrt{k^{2}+1}-1\right)$.

In what follows we will show that Criterion I can be applied to prove that \eqref{ILWP} is exactly controllable in any positive time $T>0$ and exponentially stabilizable  with any given decay rate in the Sobolev space $H_{p}^{s}(\mathbb{T}),$ with $s\in \mathbb{R}$. Indeed, first of all note that clearly,
$$
|a(k)|\leq C|k|,
$$
for some positive constant $C$ and any $k\in\mathbb{Z}$. In addition the quantity $2\pi\widehat{u}(0,t)$
is invariant by the flow of \eqref{ILWP}.

 Using the Fourier transform, it is easy to check that $(H1)$ holds with $\lambda_{k}=2\pi\left(k\sqrt{k^{2}+1}-k\right)$. See Figure \ref{fig3} for an illustrative picture. By noting that $y\mapsto 2\pi\left(y\sqrt{y^{2}+1}-y\right)$ is a strictly increasing function we then deduce that all eigenvalues $\lambda_{k}$ are simple, giving $(H2)$ and $(H3)$. Additionally, observe that $a(-k)=a(k)$ for any $k\in \mathbb{Z}$ and
 $$
 \displaystyle{\lim_{|k|\rightarrow +\infty}| (k+1) a(k+1)- k a(k)|=+\infty}.
 $$
 Thus we may apply Remark \ref{partic1} to conclude that Theorem \ref{ControlLa} holds for any $T>0$. Consequently, Theorems  \ref{st351} and \ref{estabilization} also hold.

 \subsection{The fourth-order Schr\"odinger equation}

Here we consider the control equation associated with the linear fourth-order Schr\"odinger equation
\begin{equation}\label{4nls}
	i\partial_{t}u+\partial_{x}^2u+\mu\partial_{x}^4u=0,
\end{equation}
where $u$ is a complex-valued function and $\mu\neq0$ is a real constant. Equation \eqref{4nls} is the linearized version, for instance, of the fourth-order cubic nonlinear equation
\begin{equation}\label{4nlsn}
	i\partial_{t}u+\partial_{x}^2u+\mu\partial_{x}^4u+|u|^2u=0,
\end{equation}
which was introduced in \cite{kar} and \cite{kar1} to describe the propagation of intense laser beams in a bulk medium with Kerr nonlinearity when small fourth-order dispersion are taken into account. Several results concerning well-posedness for \eqref{4nlsn} may be found in \cite{fib} (see also subsequent references). Control and stabilization for \eqref{4nlsn} have already appeared in \cite{caca}.

Equation \eqref{4nls} also serves as the linear version of the more general equation
\begin{equation}\label{4nlsn1}
	i\partial_{t}u+\partial_{x}^2u+\mu\partial_{x}^4u+F=0,
\end{equation}
with
$$
F=\frac{1}{2}|u|^2u+\mu\left( \frac{3}{8}|u|^4u+\frac{3}{2}(\partial_{x}u)^2\overline{u}+|\partial_{x}u|^2u+\frac{1}{2}u^2\partial_{x}^2\overline{u}+2|u|^2\partial_{x}^2u\right),
$$
which describes the 3-dimensional motion of an isolated vortex filament embedded in an inviscid incompressible fluid filling an infinite region. Sharp results concerning local well-posedness in Sobolev spaces were proved in  \cite{huo}.

In order to set \eqref{4nls} as in \eqref{BO} we define $\mathcal{A}=i(\partial_{x}+\mu\partial_{x}^3)$, so that $a(k)=-k+\mu k^3$. Thus we may consider the equation
\begin{equation}\label{33}
	\partial_{t}u-\partial_{x}\mathcal{A}u=Gh.
\end{equation}
We promptly see that the mass is also conserved by the flow of \eqref{33} and
$$
|a(k)|\leq C|k|^3
$$
for some constant $C>0$ and any $k\in\mathbb{Z}$. Also, we easily check that $(H1)$ holds and $\lambda_{k}=-k^2+\mu k^4$. See Figure \ref{fig3}.
   \begin{figure}[h!]
	\begin{center}
		\includegraphics[width=15cm]{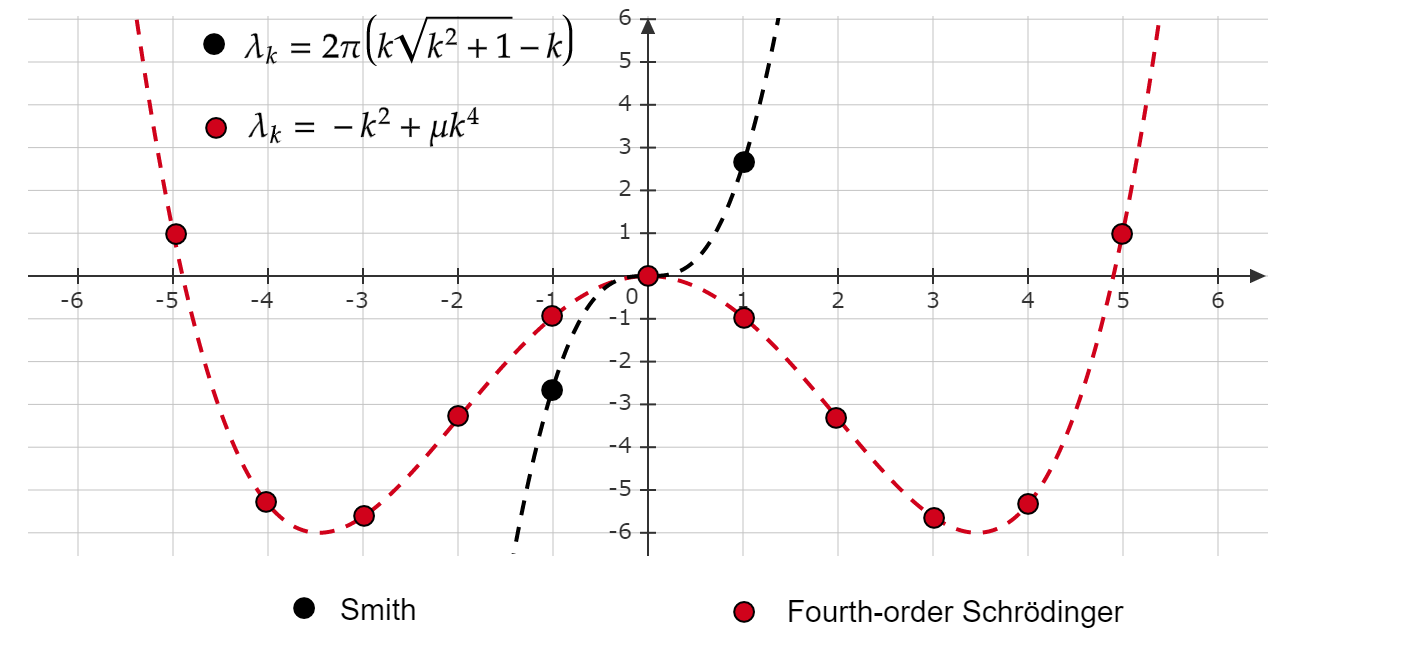}
		\caption{Dispersion of $\lambda_{k}$'s for
		the Smith and fourth-order	Schr\"odinger equations with $\mu>0.$}\label{fig3}
	\end{center}
\end{figure}
 Note if $\mu<0$ then the even polynomial $p(y)= -y^2+\mu y^4$ has no nontrivial roots, implying that $\lambda_{k}, k\neq0$ are double eigenvalues and $(H4)$ holds with $n_0=1$ and $k_1^*=1$. On the other hand, if $\mu>0$ then $p(y)$ has the nontrivial roots $\pm 1/\sqrt{\mu}$; hence, if $k_1^*$ is the less integer satisfying $1/\sqrt{\mu}\leq k_1^*$, we see that $(H4)$ holds with $n_0=4$.  

It is also clear that $(H5)$ also holds. Even more, we may check that
 $$
\displaystyle{\lim_{k\rightarrow +\infty}| (k+1) a(k+1)- k a(k)|=+\infty}.
$$
As a consequence, we   may now apply Remark \ref{partic} to conclude that Theorem \ref{ControlLag} holds for any $T>0$. Consequently, Theorems  \ref{st351} and \ref{estabilization} also hold.

\subsection{The linearized dispersion-generalized Benjamin-Ono equation} In this subsection we investigate the control and stabilization properties of the linearized dispersion-generalized Benjamin-Ono (LDGBO) equation, which contains fractional-order spatial derivatives on a periodic domain,
\begin{equation}\label{dgBO}
	\partial_{t}u+\partial_{x}D^{\alpha}u=0,\;\;\;x\in \mathbb{T},\;\;t\in\mathbb{R},
\end{equation}
where $\alpha>0,$ $u$ is a real-valued function and the Fourier multiplier operator $D^{\alpha}$ is defined as
\begin{equation}\label{OdgBO}
	\widehat{D^{\alpha}u}(k)=|k|^{\alpha}\widehat{u}(k),\;\;\;\text{for all}\;k\in \mathbb{Z}.
\end{equation}

When $\alpha\in (1,2),$ the dispersion generalized Benjamin-Ono (DGBO) equation 
\begin{equation}\label{DGBO}
	\partial_{t}u+\partial_{x}D^{\alpha}u+u\partial_{x}u=0,\;\;\;x\in \mathbb{R},\;\;t>0,
\end{equation}
defines  a family of equations which models vorticity waves in coastal zones \cite{Shrira}. The end points $\alpha=1$ and $\alpha=2$ corresponds to the well-known Benjamin-Ono and KdV equations, respectively. In this sense \eqref{DGBO} defines 
a continuum of equations of dispersive strength intermediate to two celebrated models. Regarding control and stabilization properties, the author in \cite{Flores} proved that the LDGBO equation with $\alpha \in (1,2)$ is exactly controllable in $H^{s}_{p}(\mathbb{T})$ with $s\geq 0$ and exponentially stabilizable in $L^{2}_{p}(\mathbb{T}).$
Here we extend these results to the (periodic) Sobolev space
$H^{s}_{p}(\mathbb{T})$ with $s\in \mathbb{R},$ for any $\alpha>0.$

In fact, we consider the operator  $\mathcal{A}$ in \eqref{BO} defined by $\mathcal{A}=-D^{\alpha}.$ Therefore, $a(k)=-|k|^{\alpha}$ and it is easy to verify that $$|a(k)|\leq |k|^{\alpha},$$
and
$$a(k)=a(-k),$$
for any $k\in \mathbb{Z}.$ Hence, $(H1)$ holds with $\lambda_{k}=-k|k|^{\alpha}.$ Using the L'Hospital rule we can prove that 
$$\lim_{y\to +\infty}y^{\alpha+1}\left(\left(1+\frac{1}{y}\right)^{\alpha+1}-1\right)=+\infty,\;\;\text{for any}\;\alpha>0.$$
From this, we conclude that 
$$
\displaystyle{\lim_{k\rightarrow +\infty}| (k+1) a(k+1)- k a(k)|=\lim_{k\rightarrow +\infty}\left( (k+1)^{\alpha+1}- k^{\alpha+1}\right)=  +\infty}.
$$
Thus, we can apply Remark \ref{partic1} to infer that Theorem \ref{ControlLa} holds for any $T>0$. Consequently, Theorems  \ref{st351} and \ref{estabilization} also hold in this particular case.

Finally, we point out the authors in \cite{Flores Oh and Smith} developed a dissipation-normalized Bourgain-type space, which simultaneously gains smoothing properties from the dissipation and dispersion present in the equation, to show that the nonlinear DGBO equation on a periodic setting is well-posed and local exponentially stable in $L^{2}_{p}(\mathbb{T}).$ Extending these results to the Sobolev space $H^{s}_{p}(\mathbb{T})$ with $s >0$ is a challenging task. This is an open problem.

\subsection{Higher-order Schr\"odinger equation}
In this section we consider the following higher-order Schr\"odinger equations
\begin{equation}\label{hnls1}
	i\partial_{t}u+\alpha_2\partial_{x}^2u+\alpha_4\partial_{x}^4u+\ldots+\alpha_{2m}\partial_{x}^{2m}u=0
\end{equation}
and
\begin{equation}\label{hnls}
	i\partial_{t}u+\alpha_2\partial_{x}^2u-i\alpha_3\partial_{x}^3u+\alpha_4\partial_{x}^4u-i\alpha_5\partial_{x}^5u+\ldots-i\alpha_{2m+1}\partial_{x}^{2m+1}u=0,
\end{equation}
where $m$ is a positive integer and $\alpha_j$ are real constants with $\alpha_2\neq0$,  $\alpha_{2m}\neq0$, and $\alpha_{2m+1}\neq0$. Equations \eqref{hnls1} and \eqref{hnls} are the linearized versions of an infinite hierarchy of nonlinear Schr\"odinger equations (see \cite{ank}). Thus, here we consider the control equation
\begin{equation}
	\partial_{t}u-\partial_{x}\mathcal{A}_{2m+j}u=Gh,
\end{equation}
where
$$
\mathcal{A}_{2m+j}=
\begin{cases}
	i\alpha_2\partial_{x}+i\alpha_4\partial_{x}^3+\ldots+i\alpha_{2m}\partial_{x}^{2m-1}, \quad \mbox{if}\, j=0,\\
	i\alpha_2\partial_{x}+\alpha_3\partial_{x}^2+i\alpha_4\partial_{x}^3+\alpha_5\partial_{x}^4+\ldots+\alpha_{2m+1}\partial_{x}^{2m}, \quad \mbox{if}\, j=1.
\end{cases}
$$
The symbol associated with $\mathcal{A}_{2m+j}$ is
$$
a_{2m+j}(k)=
\begin{cases}
	-\alpha_2k+\alpha_4k^3+\ldots+\alpha_{2m}(-1)^mk^{2m-1},\quad \mbox{if}\, j=0,\\
	-\alpha_2k-\alpha_2k^2+\alpha_4k^3+\alpha_5k^4+\ldots+\alpha_{2m+1}(-1)^mk^{2m} , \quad \mbox{if}\, j=1.
\end{cases}
$$
It is clear that 
$$
|a_{2m+j}(k)|\leq C|k|^{2m-1+j},
$$
for some $C>0$ and $|k|$ large enough.

Let us show that in the cases $j=0$  and $j=1$ we can apply Theorems \ref{ControlLag} and \ref{ControlLa}, respectively. Indeed, assume first $j=0$. It is easy to see that $(H1)$ holds where the eigenvalues $i\lambda_{k}$ are  such that
$$
\lambda_{k}=-\alpha_2k^2+\alpha_4k^4+\ldots+(-1)^m\alpha_{2m}k^{2m}.
$$
The polynomial $p_{2m}(y)=-\alpha_2y^2+\alpha_4y^4+\ldots+(-1)^m\alpha_{2m}y^{2m}$ is even and goes to either $+\infty$ or $-\infty$ as $|y|\to+\infty$ (according to $m$ and the sign of $\alpha_{2m})$. Thus $(H4)$ and $(H5)$ holds with $n_0=2m$ and $k_1^*$ sufficiently large. 

Assume now $j=1$. In this case we have
$$
\lambda_k=-\alpha_2k^2-\alpha_2k^3+\alpha_4k^4+\alpha_5k^5+\ldots+\alpha_{2m+1}(-1)^mk^{2m+1}.
$$
Note that the polynomial $$p_{2m+1}(y)=-\alpha_2y^2-\alpha_2y^3+\alpha_4y^4+\alpha_5y^5+\ldots+\alpha_{2m+1}(-1)^my^{2m+1}$$ has different limits ($+\infty$ or $-\infty$) as $y\to+\infty$ or $y\to-\infty$ (according to $m$ and the sign of $\alpha_{2m+1}$). Hence $(H2)$ and $(H3)$ holds with $n_0=2m+1$ and $k_1^*$ sufficiently large. \\

Furthermore, either in the case $j=0$ or $j=1,$ it can be showed that  the eigenvalues $\{i \lambda_{k}\}$ satisfies the \textquotedblleft asymptotic gap condition". Hence, we obtain the exact controllability for any  $T>0$ and Theorems  \ref{st351}$-$\ref{estabilization} hold as well.

\section{Concluding Remarks}\label{conc-rem}

In this work, we have showed two different criteria to prove that a linearized family of dispersive equations on a periodic domain is exactly controllable and exponentially stabilizable with any given decay rate in the Sobolev space $H_{p}^{s}(\mathbb{T})$ with $s\in \mathbb{R}.$  We have applied these results to prove exact controllability and exponential stabilization for the linearized Smith equation and Schr\"odinger-type equations on a periodic domain. In a forthcoming paper we plan to use these results to prove some fundamental properties like the propagation of compactness, the unique continuation property and the propagation of smoothness for the solutions of the nonlinear Smith equation in order to show that it is exactly controllable and exponentially stabilizable on a periodic domain. That is the adequate approach to prove exact controllability and exponential stabilization for nonlinear PDE's of dispersive type (see \cite{14, Laurent, Linares Rosier, Laurent Linares and Rosier, Manhendra and Francisco 2}). However, the symbol of the linear part associated to the Smith equation creates extra difficulty to prove the unique continuation property on a periodic domain.  This work is in progress.

% -----------------------------------------------------------

\subsection*{\textbf{Acknowledgment}}
F. J. Vielma Leal is partially supported by FAPESP/Brazil grant 2020/14226-4. A. Pastor is partially supported by  FAPESP/Brazil grant 2019/02512-5 and CNPq/Brazil grant 303762/2019-5.
The authors would like to thank Prof. Roberto Capistrano-Filho for many helpful discussions and suggestions to complete this work.

%%-----------------------------
%%      your bibliography
%%-----------------------------


\begin{thebibliography}{99}
\small




\bibitem{Abdelouhab Bona Felland and Saut} L. Abdelouhab, J. L. Bona, M. Felland, J-C. Saut, {\em Nonlocal models for nonlinear, dispersive waves}, Phys. D. {\bf 40}  (1989) 360-392.

\bibitem{ank} A. Ankiewicz, D. J. Kedziora, A. Chowdury, U. Bandelow, N. Akhmediev, {\em Infinite hierarchy of nonlinear Schr\"odinger equations and their solutions}, Phys. Rev. E. {\bf 93}  (2016), 012206.



\bibitem{Ball and Slemrod} J. M. Ball, M. Slemrod, {\em Nonharmonic Fourier series and the stabilization of
distributed semi-linear control systems}, Comm. Pure Appl. Math. {\bf 32} 4 (1979) 555--587.

\bibitem{caca} R. Capistrano-Filho, M. Cavalcante,  {\em Stabilization and control for the biharmonic Schr\"{o}dinger equation}, Appl. Math. Optim.  (2019), https://doi.org/10.1007/s00245-019-09640-8.

\bibitem{Capistrano  y Andressa} R. Capistrano-Filho, A. Gomes, {\em Global control aspects for long waves in nonlinear dispersive media}, preprint arXiv:2013.00921v1.

\bibitem{3} T. Cazenave, H. Haraux, {\em An introduction to Semilinear Evolutions Equations}, John Wiley and Sons, Inc., Revised edition, Clarendon Press-Oxford (1998).

\bibitem{Coron} J.-M. Coron, {\em Control and Nonlinearity}, in: Mathematical surveys and Monographs, vol. 136, Amer. Math. Soc., 2007.


\bibitem{Coron Crepau} J.-M. Coron, E. Cr\'epeau, {\em Exact boundary controllability of a nonlinear KdV equation
with a critical length}, J. Eur. Math. Soc. {\bf 6} (2004) 367--398.


\bibitem{Dehman Gerard Lebeau} B. Dehman, P. G\'{e}rard, G. Lebeau, {\em Stabilization and control for the nonlinear Schr$\ddot{o}$dinger equation on a compact surface}, Math. Z. {\bf 254} (2006) 729--749.


\bibitem{fib} G. Fibich, B. Ilan, G. Papanicolaou,  {\em Self-focusing with fourth-order dispersion},  SIAM J. Appl. Math. {\bf 62} (2002), 1437--1462.


\bibitem{Flores} C. Flores, {\em Control and stability of the linearized dispersion-generalized Benjamin-Ono equation on a periodic domain},  Math. Cont. Sig. Systems {\bf 3} (2018),  Art. 13, 16 pp. 

\bibitem{Flores Oh and Smith} C. Flores, S. Oh, D. Smith, {\em Stabilization of dispersion-generalized Benjamin-Ono}, arXiv:1709.10224v1  (2017).



\bibitem{9} C. Heil, {\em A Basis Theory Primer}, Expanded Edition, Applied and Numerical Harmonic Analysis, Birkhauser, Birkhäuser/Springer, New York, 2011.

\bibitem{huo} Z. Huo, Y. Jia, {\em A refined well-posedness for the fourth-order nonlinear Schr\"odinger equation related to the vortex filament}, Com. Part. Dif. Eq. {\bf 32} (2007), 1493--1510.

\bibitem{8} A. E. Ingham, {\em Some trigonometrical Inequalities with applications in the theory of series}, Math. Z. {\bf 41} (1936), 367--379.

\bibitem{iorio} R. J. Jr. Iorio, {\em KdV, BO and friends in weighted Sobolev spaces}, Functional-analytic methods for partial differential equations (Tokyo, 1989), 104--121, Lecture Notes in Math., 1450, Springer, Berlin, 1990.

\bibitem{6} R. J. Jr. Iorio, V. Magalh\~{a}es, {\em Fourier Analysis and Partial Differential Equations}, Cambrige Universiy Press (2001).

\bibitem{kar} V.I. Karpman, {\em Stabilization of soliton instabilities by higher-order dispersion: fourth order nonlinear Schr\"odinger-type equations},  Phys. Rev. E {\bf 53} (1996), 1336--1339.

\bibitem{kar1}  V. I. Karpman,  A.G. Shagalov, {\em Stability of soliton described by nonlinear Schrödinger type equations with higher-order dispersion},  Physica D \textbf{144} (2000), 194--210.



\bibitem{7} V. Komornik,  P. Loreti, {\em Fourier Series in Control Theory}, Springer Monographs in Mathematics (2005).

\bibitem{Laurent Camille} C. Laurent, {\em Internal control of the Schr\"odinger equation},  Math. Cont. $\&$ Related Fields,  {\bf 4} 2 (2014) 161-186.

\bibitem{Laurent} C. Laurent, {\em Global controllability and stabilization for the nonlinear Schr\"odinger equation on an interval}, ESAIM Control Optm. Cal. Var. {\bf 16} 2 (2010) 356--379.

\bibitem{14} C. Laurent, L. Rosier, B. -Y. Zhang, {\em Control and stabilization of the Korteweg-de Vries equation on a periodic domain}, Comm. Part. Dif. Eq.,  {\bf 35} 4 (2010) 707--744.




\bibitem{Laurent Linares and Rosier} C. Laurent, F. Linares, L. Rosier, {\em Control and Stabilization of
the Benjamin-Ono Equation in $L^{2}(\mathbb{T})$}, Arch. Rational Mech. Anal. {\bf 218} 3 (2015) 1531--1575.





\bibitem{1} F. Linares, J. H. Ortega, {\em On the controllability and stabilization of the linearized Benjamin-Ono equation}, ESAIM: Cont. Optm. Cal. Var. {\bf 11} 2 (2005), 204--218.

\bibitem{Linares Rosier} F. Linares, L. Rosier, \textit{Control and Stabilization of the Benjamin-Ono Equation on a Periodic Domain}, Trans.  Amer. Math. Soc., {\bf 367} 7 (2015) 4595--4626.



\bibitem{Liu} K. Liu, \textit{ Locally distributed control and damping for the conservative systems}, SIAM J. Cont. Optim., 35 (1997), 1574-1590.



\bibitem{Menzala Vasconcellos Zuazua} G. P. Menzala, C. F. Vasconcellos, E. Zuazua, {\em Stabilization of the Korteweg-de Vries
equation with localized damping}, Quart. Appl. Math., {\bf 60} 1 (2002) 111--129.

\bibitem{Micu Ortega Rosier and Zhang} S Micu, J. Ortega, L. Rosier, B-Y. Zhang,{\em Control and Stabilization of a family
of Boussinesq Systems}, Disc. and Cont. Dyn. Syst. {\bf 24} 2 (2009) 273--313.



\bibitem{Manhendra and Francisco} M. Panthee, F. Vielma Leal,  {\em On the controllability and stabilization of the linearized Benjamin equation on a periodic domain}, Nonlinear Analysis.: Real World Applications, {\bf 51} (2020) 102978.

\bibitem{Manhendra and Francisco 2} M. Panthee, F. Vielma Leal,  {\em On the controllability and stabilization of the Benjamin equation on a periodic domain}, Annales de l'Institut Henri Poincar\'{e}-Analyse non lin\'{e}aire, \text{http://doi.org/10.1016/j.anihpc.2020.12.004}




\bibitem{4} A. Pazy, {\em Semigroups of Linear Operators and Applications to Partial Differential Equations}, Springer-Verlag, New York Inc (1983).

\bibitem{Rosier Lionel Zhang} L. Rosier, B. -Y. Zhang, {\em Local exact controllability and stabilizability of the nonlinear Schr$\ddot{o}$dinger equation on a bounded interval}, SIAM J. Cont. Optim., {\bf 48} 2 (2009) 972-992.

\bibitem{Rosier Lionel Zhang 2} L. Rosier, B. -Y. Zhang, {\em Control and stabilization of the nonlinear Schr$\ddot{o}$dinger equation on rectangles}, Math. Models and Meth. in App. Sciences, {\bf 20} 12 (2010) 2293-2347.


\bibitem{Rosier 1} L.  Rosier, {\em Exact boundary controllability for the Korteweg-de Vries equation on a bounded
domain}, ESAIM: Cont. Optm. Calc. Var., {\bf 2} (1997) 33--55.

\bibitem{Rosier and Zhang 2} L. Rosier, B. -Y.Zhang, {\em Global stabilization of the generalized Korteweg-de Vries equation}, SIAM J. Cont. Optim., {\bf 45} 3 (2006) 927--956.

\bibitem{Rudin} W. Rudin, {\em Functional Analysis}, Mc-Graw Hill, Inc. Second Edition (1991).

\bibitem{Russell} D. L. Russell, {\em Controllability and stabilizability theory for linear partial differential equations: recent progress and open questions},
SIAM Review, {\bf 20} 4 (1978) 639--739.

\bibitem{Russell and Zhang} D. L. Russell, B. -Y. Zhang, {\em  Controllability and stabilizability of the thrid-order linear dispersion equation on a periodic domain}, SIAM J. Cont. and Optm,  {\bf 31} 3 (1993) 659--676.

\bibitem{10} D. L. Russell, B. -Y. Zhang, {\em Exact controllability and stabilizability of the Korteweg-de Vries equation}, Trans. Amer. Math. Soc.,  {\bf 348} 9 (1996) 3643--3672.


\bibitem{Shrira} V. I. Shrira, V. V. Voronovich, {\em
 Nonlinear dynamics of vorticity waves in the coastal zone}. J Fluid Mech. {\bf 326} (1996) 181--203.


\bibitem{Slemrod} M. Slemrod,   \textit{A note on complete controllability and stabilizability for linear control systems in Hilbert space}, SIAM J. Control,  {\bf 12} 3 (1974) 500--508.


\bibitem{Smith} R. Smith, \textit{Nonlinear Kelvin and continental-shelf waves}, J. Fluid Mech. {\bf 57} (1972) 379-391.



\bibitem{Zhang 1} B. -Y. Zhang, {\em Exact boundary controllability of the Korteweg-de Vries equation}, SIAM J. Cont.
Optm., {\bf 37} 2 (1999), 543--565.
\end{thebibliography}
\end{document}